\documentclass[aop]{imsart}
\usepackage{amsmath, amssymb, amsthm, amsfonts}
\usepackage[numbers]{natbib}
\usepackage{hyperref}

\numberwithin{equation}{section}

\startlocaldefs
\theoremstyle{plain}
\newtheorem{thm}{Theorem}[section]
\newtheorem{lemma}[thm]{Lemma}
\newtheorem{prop}[thm]{Proposition}
\newtheorem{cor}[thm]{Corollary}
\theoremstyle{definition}
\newtheorem{defn}[thm]{Definition}
\newtheorem{example}[thm]{Example}
\newtheorem{rem}[thm]{Remark}

\DeclareMathOperator{\diag}{diag} 
\DeclareMathOperator{\Tr}{\rm Tr} 
\DeclareMathOperator{\cov}{Cov} 
\newcommand{\CC}{\mathbb C} 
\newcommand{\RR}{\mathbb R} 
\newcommand{\ZZ}{\mathbb Z} 
\newcommand{\EE}{\mathbb E} 
\newcommand{\NN}{\mathbb N} 
\newcommand{\PP}{\mathbb P} 
\newcommand{\norm}[1]{\left\lVert #1 \right\rVert} 
\newcommand{\GT}{ \mathrm{GT}_M} 
\newcommand{\abs}[1]{\left\lvert #1 \right\rvert} 
\newcommand{\Ploc}{\mathbb{P}_N^{\Sigma, \mathrm{loc}}}
\newcommand{\Zloc}{Z_N^{\Sigma, \mathrm{loc}}} 
\newcommand{\one}{\mathbf 1} 

\endlocaldefs

\begin{document}

\begin{frontmatter}
\title{Markovian Random Words, a Gaussian Random Matrix Model and Beyond}

\begin{aug}

\author[A]{\fnms{Christian}~\snm{Houdr\'e}\ead[label=e1]{houdre@math.gatech.edu}}
\author[A]{\fnms{Dorian}~\snm{Kidonakis}\ead[label=e2]{doriankidonakis@gmail.com}}

\address[A]{School of Mathematics, Georgia Institute of Technology\printead[presep={ ,\ }]{e1,e2}}

\end{aug}

\begin{abstract}
After centering and scaling, the length of the longest weakly increasing subsequence of a finite-state Markov random word converges to a Brownian last-passage functional with dependent components. We investigate when this functional, and more generally the limiting RSK Young-diagram shape, admits a Gaussian random-matrix representation. We introduce Hermitian Gaussian matrices with correlated diagonal entries and show that isotropy of the Brownian covariance on the trace-zero hyperplane yields such a representation. Within this natural class of Hermitian models with independent Gaussian off-diagonal entries, we prove a rigidity theorem showing that isotropy, together with equality of the off-diagonal variances, characterizes the conditional largest-eigenvalue representation. Explicit three-letter non-cyclic Markov chains demonstrate that this representation is not universal: the
limiting Brownian functional may fail to have the law of the largest  eigenvalue of {\it any} centered jointly Gaussian $3\times 3$ Hermitian matrix. Finally, we extend the random-matrix identification to chains with non-uniform stationary distributions and to countable state spaces.

\end{abstract}



\end{frontmatter}

\section{Introduction and Main Results}

The limiting law of the length of the longest weakly increasing subsequence of a random word generated by a Markov chain on a finite totally ordered state space was identified in \cite{HL2} as a maximal functional of a multidimensional Brownian motion with dependent components. In the particular case of cyclic Markov chains on alphabets of cardinality $M=2,3$, this functional is, up to a scaling factor, the same as for words with iid uniformly distributed letters \cite{CG,HL1,HL2}. Therefore, by the results of \cite{Baryshnikov,GTW,BGH}, it has the same law as the largest eigenvalue of an $M\times M$ traceless GUE matrix, thereby answering positively, in these cases, a question raised in \cite{Kuperberg} (Conjecture 7). For $M\ge 4$, however, even in the cyclic setting, the limiting Brownian functional is no longer necessarily of traceless-GUE type. This naturally raises the question of whether or not the limiting laws associated with Markov random words admit a Gaussian random-matrix representation.

Our interest in this question originates in the classical connection between increasing subsequences, Young diagrams, and random matrices. For iid uniformly distributed letters on an alphabet of cardinality $M$, the centered and normalized length of the longest weakly increasing subsequence converges to the largest eigenvalue of an $M\times M$ traceless GUE matrix. More generally, the normalized limiting shape of the Young diagram obtained through the Robinson–Schensted–Knuth (RSK) correspondence is distributed as the spectrum of such a matrix; see \cite{Kerov,TW,J,Baryshnikov,GTW,Bu,Me}. 
The interlacing eigenvalues of successive GUE principal minors form the GUE-minor process; see, in particular, \cite{NJEJP06}.

These connections belong to the broader relationship between increasing subsequences, Young diagrams, non-intersecting paths, and random matrices; see also \cite{J1,J2}.
Extensions to non-uniform random words were obtained in \cite{ITW1,ITW2,HX}. 
The connection between non-colliding Brownian motions, Weyl chambers, and random matrices is part of a broader circle of results; see, among others, \cite{Baryshnikov, Grabiner, BJ, BBO, O, OY}.  Closely related Brownian last-passage functionals also arise naturally in tandem-queueing systems; see, for example, \cite{GW}. The representation-theoretic background also has antecedents in atomic physics, through studies of symmetry in superradiance and subradiance \cite{CLPP} and of permutation-symmetry probabilities for tensor products of identical atomic states \cite{ARS}. The connection between tensor-product decompositions, random-word RSK shapes, and Gaussian random matrices is developed explicitly in \cite{Kuperberg}.

The general Markovian limits obtained in \cite{HL2} involve Brownian motions with nontrivial covariance structures, and the classical GUE correspondence does not directly extend to this setting. To investigate this issue, we introduce a class of Gaussian Hermitian random matrices whose off-diagonal upper-triangular entries are independent circular complex Gaussian random variables while the diagonal entries form a centered Gaussian vector with a prescribed covariance matrix $\Sigma$. 
These ensembles are generally not invariant under arbitrary unitary
conjugation, although full unitary invariance persists for certain choices of  
$\Sigma\neq I$.  We derive their eigenvalue distributions, describe their residual conjugation symmetries, and construct discrete RSK models whose Gelfand--Tsetlin patterns converge to the spectra of the principal minors of these matrices. This provides a probabilistic extension, to a dependent setting, of the approach developed in \cite{Baryshnikov,GTW,BGH}.

A central point of the paper is that full unitary invariance is not the relevant condition for the Brownian/random-matrix correspondence. Let
$Q=I-\frac1M{\bf 1}{\bf 1}^{T}$ be the orthogonal projection onto the trace-zero hyperplane ${\bf 1}^\perp=\{x\in\mathbb R^M: {\bf1}^T x=0\}$.  (Throughout, $\one$ is the column vector of $\RR^M$ with each coordinate equal to $1$, while $I$ is the $M\times M$ identity matrix.)  
The limiting Brownian motion $B^\Sigma$ has covariance ${\rm Cov} (B^\Sigma(t), B^\Sigma(s)) = (t\wedge s) \Sigma$, and hence 
$\Sigma$ is the covariance matrix of $B^\Sigma(1)$.  The relevant condition turns out to be $Q\Sigma Q=cQ$, 
for some $c\geq0$. Thus, what matters is isotropy of the covariance after projection onto ${\bf 1}^\perp$, namely that $Q\Sigma Q$ acts as $cI$ on ${\bf 1}^\perp$, rather than isotropy of the full covariance matrix. Under this condition the projected Brownian motion has, up to the factor $\sqrt c$, 
the same law as standard Brownian motion projected 
onto ${\bf 1}^\perp$, and the traceless-GUE structure reappears.

We show, moreover, that this condition has a rigidity character in the nondegenerate case. Decomposing the Brownian motion into its endpoint and an independent Brownian bridge leads naturally to a conditional last-passage functional, with the endpoint playing the role of the diagonal of the corresponding Gaussian matrix. We prove that this conditional Brownian functional admits, for every $d\in\RR^M$, a representation as
$\lambda_{\max}\bigl(\operatorname{diag}(d)+H\bigr)$, 
where $H$ has zero diagonal and independent complex Gaussian off-diagonal entries of variances $v_{ij}$, if and only if there exists $c>0$ such that
$Q\Sigma Q=cQ$, and $v_{ij}=c$, $i<j$.  Thus the deterministic vector $d$ appearing in the rigidity theorem should be understood as the value of the random endpoint after conditioning, rather than as a deterministic drift in the original Markovian problem. The necessity is obtained by reducing the identity successively to two- and three-dimensional subsystems: the two-dimensional reductions determine the individual off-diagonal variances from the Brownian covariance, while the three-dimensional reductions force all these variances to coincide.

The Gaussian random-matrix representation is nevertheless not universal. We construct explicit Markov chains for which the limiting Brownian functional cannot have the law of the largest eigenvalue of \emph{any} centered jointly Gaussian Hermitian matrix whose dimension equals the alphabet size, even when arbitrary dependence among the matrix entries is allowed. The obstruction is based on small-deviation probabilities and the geometry of Brownian motion in cones. In a four-state cyclic example, the limiting Brownian functional $F_\Sigma$ satisfies
$\mathbb P(F_\Sigma \leq\varepsilon)=O(\varepsilon^{17})$, as $\varepsilon\downarrow0$, 
which is incompatible with the fact that the real dimension of the space of $4\times4$ Hermitian matrices is $16$.  
Indeed, this is incompatible with the Gaussian small-ball lower bound of order $\varepsilon^r$, where $r\le 16$ is the dimension of the linear support of the Gaussian matrix. 

The phenomenon is not restricted to alphabets of cardinality at least four. For three-state chains with uniform stationary distribution, let
$X_1=B_1^\Sigma-B_2^\Sigma$, $X_2=B_2^\Sigma-B_3^\Sigma$, where $(B_1^\Sigma, B_2^\Sigma, B_3^\Sigma)^T$ is the (centered) three-dimensional Brownian motion with covariance $\Sigma$, 
and let $\rho$ denote the correlation of $X_1(1)$ and $X_2(1)$. We show that whenever $-1<\rho<-1/2$, 
the corresponding limiting Brownian functional cannot be the largest  eigenvalue of any centered jointly Gaussian $3\times3$ Hermitian matrix. By contrast, in the nontrivial  \emph{cyclic} three-state case one has $\rho=-1/2$, and the traceless-GUE representation is recovered. Thus, the distinction is not simply one of alphabet size: the underlying covariance geometry determines whether a Gaussian random-matrix representation can persist.

The paper is organized as follows. Section 2 introduces the Gaussian Hermitian matrix model with correlated diagonal entries, derives its eigenvalue law and conjugation symmetries, and develops the discrete RSK and Gelfand--Tsetlin framework. Section 3 establishes the convergence of these discrete models to the Gelfand--Tsetlin patterns of the Gaussian matrix model, including positive semidefinite covariance matrices by approximation. Section 4 applies this framework to Markov random words, identifies the condition $Q\Sigma Q=cQ$, proves the rigidity theorem for largest-eigenvalue representations, and characterizes the corresponding Markov transition matrices. Section 5 gives explicit three- and four-state counterexamples showing that outside this structure a Gaussian random-matrix representation may fail altogether.  Section 6 extends the random-matrix representation to Markov chains with non-uniform stationary distributions, yielding blockwise Gaussian random-matrix limits for the rows of the associated RSK Young diagrams.  This approach is then used to tackle, via a sandwiching argument, the infinite countable alphabet case.  The paper concludes with a final remark presenting some perspectives and open questions connected to the previous results.

\section{Correlated-Diagonal Gaussian Ensembles and Discrete RSK Models}

This section develops the random-matrix and discrete RSK framework used in the sequel. We first introduce the Gaussian Hermitian ensemble $\mathcal M_\Sigma$, in which the diagonal vector has covariance $\Sigma$, and study its eigenvalue distribution and conjugation symmetries. We then formulate the model at the level of Gelfand--Tsetlin patterns and construct, through matrix RSK applied to independent geometric arrays, a discrete approximation of these patterns. The key ingredient is a localized exponential change of measure from the standard geometric-RSK model, chosen as the discrete counterpart of the change of measure from the GUE to $\mathcal M_\Sigma$. Section~3 shows that, after centering and scaling, this discrete model indeed converges to the Gelfand--Tsetlin pattern of ${\cal M}_\Sigma$.

Let $H = (h_{ij})_{1\le i,j \le M}$ be an 
$M \times M$ Hermitian random matrix where the upper-triangular off-diagonal 
entries are iid with distribution $N(0,1)_{\CC}$ (with real
and imaginary parts $\Re h_{ij}\sim N(0, 1/2)$ and $\Im h_{ij}\sim  N(0,1/2)$) while the diagonal elements form a centered Gaussian vector ${\bf d} = {\bf d}(H)$ which is independent of the off-diagonal
elements, ${\bf d}(H)=(h_{11},\ldots,h_{MM})^T\sim\mathcal N(0,\Sigma)$. 
Let ${\cal M}_\Sigma$ denote the ensemble of such random matrices, and so when 
$\Sigma=I$, this is the standard GUE under our normalization.

To continue with our notation, the set of $M\times M$ Hermitian matrices is denoted by ${\cal H}_M$; it is a finite-dimensional real Hilbert space, with respect to the Hilbert-Schmidt inner product $\langle X, Y\rangle _{HS} = \Tr (XY^*)$.  
The entrywise Lebesgue measure 
is $dH = \prod_{i=1}^M dh_{ii}\prod_{1 \leq i <j \leq M} d^2 h_{ij}$, where $d^2 h_{ij} = d(\Re h_{ij}) d (\Im h_{i
j})$, while $dg$ denotes normalized Haar probability measure on $U(M)$, the unitary group of degree $M$.

A non-identity covariance matrix $\Sigma$ for the diagonal vector generally destroys the full unitary invariance of the GUE and we now describe the remaining symmetries. Indeed, $\Sigma$ is the covariance matrix of the vector
$(h_{11},\ldots,h_{MM})\in\mathbb R^M$, 
and is not the covariance operator of $H$, viewed as a Gaussian random element of the real Hilbert space ${\cal H}_M$. Therefore, under conjugation $H\mapsto UHU^*$, the covariance $\Sigma$ does not in general transform according to $\Sigma\mapsto U\Sigma U^*$ for a general unitary matrix $U$; the diagonal entries of $UHU^*$ involve both the diagonal and off-diagonal entries of $H$. 

The full symmetry group of the ensemble ${\cal M}_\Sigma$ should therefore be 
defined directly by
$$
G_\Sigma = \left\{U\in U(M):UHU^*\stackrel{d}{=}H
\text{ for }H\in {\cal M}_\Sigma\right\}, 
$$
where $\stackrel{d}{=}$ indicates equality in distribution.   
Certain symmetries are immediate. First, every diagonal unitary matrix
$$
D=\operatorname{diag}(e^{i\theta_1},\ldots,e^{i\theta_M}), 
$$ 
belongs to $G_\Sigma$. Indeed, conjugation by $D$ leaves the diagonal vector unchanged and maps
$$
h_{ij}\longmapsto e^{i(\theta_i-\theta_j)}h_{ij},
\qquad i\ne j.
$$
Since the off-diagonal entries are independent circular complex Gaussian random variables, their joint law is unchanged.

Next, let $P_\sigma$ be the permutation matrix associated with a permutation $\sigma\in\mathfrak S_M$, the symmetric group of $\{1, 2, \dots, M\}$. Then
$$
P_\sigma H P_\sigma^*
\stackrel d=H, 
$$
whenever
$P_\sigma\Sigma P_\sigma^T=\Sigma$.  
Indeed, the off-diagonal iid Gaussian law is invariant under simultaneous permutation of rows and columns, while the diagonal vector is transformed into $P_\sigma {\bf d}(H)$, whose covariance is $P_\sigma\Sigma P_\sigma^T$.

Thus, $G_\Sigma$ always contains the diagonal torus $U(1)^M$, together with the permutations preserving $\Sigma$. In particular, it contains the corresponding semidirect product
$$ 
U(1)^M\rtimes
\{\sigma\in\mathfrak S_M:
P_\sigma\Sigma P_\sigma^T=\Sigma\}=U(1)^M\rtimes \operatorname{Aut}_{\rm perm}(\Sigma).$$

For a generic covariance matrix $\Sigma$, these are the natural residual symmetries of the model. Additional continuous unitary symmetries may occur for special choices of $\Sigma$, but they must be determined from the covariance operator of the entire Gaussian matrix $H$, rather than from the commutant of $\Sigma$ as an $M\times M$ matrix.

The ordinary GUE is recovered when the covariance of the diagonal entries agrees with the variance dictated by the isotropic Gaussian structure on ${\cal H}_M$. Under the normalization used here, this is $\Sigma=I$. In this case the density of $H$ depends only on $
\operatorname{Tr}(H^2)$, 
and therefore $$
UHU^*\stackrel d=H
\qquad\text{for every }U\in U(M).
$$
Hence $G_I=U(M)$. 

The condition $\Sigma=I$ characterizes the standard GUE within our
normalization, but it is not necessary for full unitary invariance.
For example, let $H_0\in\mathcal M_I$ and let $Z\sim N(0,b)$,
$b>0$, be independent of $H_0$. Then
\[
H:=H_0+ZI\in\mathcal M_{I+b\mathbf1\mathbf1^T}.
\]
Indeed, its off-diagonal entries are unchanged, while
$\operatorname{Cov}(d(H))=I+b\mathbf1\mathbf1^T$.
Moreover, for every $U\in U(M)$,
\[
UHU^*=UH_0U^*+ZI\stackrel d=H_0+ZI=H.
\]
Thus $G_{I+b\mathbf1\mathbf1^T}=U(M)$.
These examples lie outside the hypotheses of Proposition~2.2,
since $I+b\mathbf1\mathbf1^T$ acts as the identity on
$\mathbf1^\perp$.

For general $\Sigma$, however, full unitary invariance is lost.
As will become important in Section~4, such invariance is not
necessary for the Brownian last-passage functional to admit a
Gaussian random-matrix representation. The relevant condition
there is instead $Q\Sigma Q=cQ$, for some $c\ge 0$.

\medskip

To begin our study, we derive, when $\Sigma$ is invertible, the 
joint probability density function of the eigenvalues $\mu_1, \dots, \mu_M$ of the elements of ${\cal M}_\Sigma$. 
(Below, $\Delta(\mu) = \prod_{1\leq i<j\leq M} (\mu_i-\mu_j)$ is the usual Vandermonde determinant.)  The convergence results below are all that is needed to handle singular covariance matrices via approximations with positive definite ones.  However, for a singular covariance matrix $\Sigma$, a representation for the law of the eigenvalues akin to \eqref{densityinvert} can also be obtained, in the sense of distributions, and involves the Moore--Penrose pseudoinverse and pseudodeterminant.  The convergence results extend to singular covariance matrices by approximation with positive definite covariance matrices, as established in Theorem~\ref{thm:lemma21singgen}, without requiring a density formula in the singular case.

\begin{prop}\label{prop:JPDF}
Let $\Sigma$ be invertible and let $H \in {\cal M}_\Sigma$.  The joint probability density function of the ordered eigenvalues $\mu_1\geq \ldots \geq \mu_M$ 
of $H$ is

\begin{equation}\label{densityinvert}
f_{\Sigma}(\mu) = \frac{M!\Delta(\mu)^2}{\sqrt{(2\pi )^M \det \Sigma}\prod_{j=1}^M j!} \exp{ \left ( - \frac{1}{2} \sum_{i=1}^M \mu_i^2 \right )} \mathcal{I}(J,\Sigma),
\end{equation}
where
\begin{equation*}
\mathcal{I}(J,\Sigma) = \int_{U(M)} \exp \left ( \frac{1}{2} \mathbf v (U)^T(I - \Sigma^{-1}) \mathbf v (U) \right ) dg(U), 
\end{equation*}
with $J = \diag ( \mu_1, \ldots, \mu_M)$ and where 
$\mathbf v (U) \in \RR^M$ is the vector of diagonal
elements of $UJ U^*$.  
\end{prop}
\begin{proof}
Let $H \in {\cal M}_\Sigma$.  Since the (column) vector of diagonal
elements ${\bf d} = {\bf d}(H)$ is independent of the off-diagonal
entries, and the off-diagonal entries are iid, the density of the entries is 
  
\begin{equation*}
p_\Sigma(H) = \frac{\exp \left (- \frac{1}{2} \mathbf{d}^T \Sigma^{-1}\mathbf{d} \right )}{\sqrt{(2\pi)^M \det \Sigma}} \prod_{1 \leq i < j \leq M} \frac{\exp \left ( - |h_{ij}|^2\right)}{\pi}.
\end{equation*}

\noindent
Since $\displaystyle \Tr(H^2) = \sum_{i=1}^M h_{ii}^2 + 2 \sum_{1\leq i < j \leq M} |h_{ij}|^2$, 
\begin{equation*}
p_\Sigma(H) = \frac{1}{\sqrt{2^M \pi^{M^2} \det \Sigma}} \exp \left (\frac{1}{2} \mathbf d^T (I - \Sigma^{-1}) \mathbf{d} - \frac{1}{2} \Tr(H^2) \right ).
\end{equation*}
By the Weyl integration formula \cite[Theorems 2.8 and 2.10]{VA} (with the above normalization of $dH$ and normalized Haar measure $dg$), for $F$ nonnegative Borel,
\begin{equation}
\label{WeylIntegration}
\int_{\mathcal{H}_M} F(H) dH = \frac{\pi^{M(M-1)/2}}{\prod_{j=1}^M j!} \int_{\RR^M} \int_{U(M)} F(UJU^*) \Delta (\mu)^2 dg(U) d \mu.
\end{equation}
By linearity, the same holds for every integrable Borel $F$; apply \eqref{WeylIntegration} to $F_+$ and $F_-$, and then to complex-valued functions.
By cyclicity of the trace, $\Tr((UJU^*)^2) = \Tr(J^2) = \sum_{i=1}^M \mu_i^2$, so
\begin{equation*}
p_\Sigma(UJU^*) = \frac{1}{\sqrt{2^M \pi^{M^2} \det \Sigma}} \exp \left ( - \frac{1}{2} \mathbf{v} (U)^T(\Sigma^{-1}-I) \mathbf{v}(U) - \frac{1}{2}\sum_{i=1}^M \mu_i^2 \right ).
\end{equation*}

Let $\varphi$ be a bounded Borel function on the ordered Weyl chamber
$\mathcal W_M=\{\mu\in\mathbb R^M:\mu_1\ge\cdots\ge\mu_M\}$, 
and let $S$ denote its symmetric extension to $\mathbb R^M$, namely
$S(\mu)=\varphi(\mu^\downarrow)$,
where $\mu^\downarrow$ denotes the decreasing rearrangement of $\mu$. If $\mu(H)$ denotes the ordered eigenvalue vector of $H$, set
$F(H)=\varphi(\mu(H))p_\Sigma(H)$.  
For $J=\operatorname{diag}(\mu_1,\ldots,\mu_M)$, the ordered eigenvalue vector of $UJU^*$  is $\mu^\downarrow$; hence $\varphi(\mu(UJU^*))=S(\mu)$.
Then, the right-hand side of \eqref{WeylIntegration} is
\begin{equation*}
\int_{\mathcal H_M} \varphi(\mu(H)) p_\Sigma(H) dH = \int_{\RR^M} {S}(\mu) \!\left(\frac{\Delta(\mu)^2\exp \left (- \frac{1}{2} \sum_{i=1}^M \mu_i^2 \right )}{\sqrt{(2\pi )^{M} \det \Sigma}\prod_{j=1}^M j!}  \mathcal{I}(J,\Sigma) \right)\! d\mu.
\end{equation*}

From the right invariance of the Haar measure, the unitary integral is symmetric under permutations of the entries of $\mu$ (if $P$ is the corresponding permutation matrix, then $PJP^*$ permutes the entries of $\mu$, and the substitution $U\mapsto UP$ in Haar measure proves the symmetry) and since $\Delta(\mu)^2$ and $\sum_{i=1}^M \mu_i^2$ are symmetric in $\mu$, so is the whole integrand. Observe that $\RR^M$ is the union of the $M!$ Weyl chambers up to a set of Lebesgue measure zero. Thus, restricting to the chamber $\{\mu_1 \geq \ldots \geq \mu_M\}$ introduces a factor of $M!$. We conclude that the joint density of the ordered eigenvalues $\mu_1 \geq \ldots \geq \mu_M$ is
\begin{equation*}
f_{\Sigma}(\mu) = \frac{M!\Delta(\mu)^2\exp{ \left ( - \frac{1}{2} \sum_{i=1}^M \mu_i^2 \right )} }{\sqrt{(2\pi )^{M} \det \Sigma}\prod_{j=1}^M j!}  \mathcal{I}(J,\Sigma).
\end{equation*}
\end{proof}

With the help of Proposition~\ref{prop:JPDF}, the following delineates conjugation symmetries of ${\cal M}_\Sigma$ and how the GUE symmetry breaks down in the ensemble introduced above.  Below, recall that 
the conjugation symmetry group of ${\cal M}_\Sigma$ is 
$$ 
G_\Sigma :=
\left\{U\in U(M): UHU^*\stackrel{d}{=}H\right\}, 
$$ 
where $\stackrel{d}{=}$ indicates equality in distribution.  

\begin{prop}\label{prop:symmetry-group}
Let $\Sigma$ be positive definite, let $H\in{\cal M}_\Sigma$ and let $1\notin\operatorname{spec}(\Sigma)$.  
Then, 
\begin{equation}\label{sym}
G_\Sigma =\left\{DP:
D\in U(1)^M, P \text{ is a permutation matrix such that}\ 
P\Sigma P^T=\Sigma \right\}.
\end{equation}
Equivalently, $G_\Sigma
\simeq U(1)^M\rtimes\operatorname{Aut}_{\rm perm}(\Sigma)$, 
where
\begin{equation*}
\operatorname{Aut}_{\rm perm}(\Sigma)=
\left\{P:\ P\text{ is a permutation matrix and }P\Sigma P^T=\Sigma
\right\}.
\end{equation*}
\end{prop}

\begin{proof}
From the density of the entries of $H$,
$$
p_\Sigma(H)=C_\Sigma
\exp\left\{
-\frac12\operatorname{Tr}(H^2)+\frac12 {\bf d}(H)^T(I-\Sigma^{-1}){\bf d}(H)
\right\},
$$
where $C_\Sigma>0$ is a normalizing constant.  Set $A:=I-\Sigma^{-1}$.  
Since $1$ is not an eigenvalue of $\Sigma$,
$A=\Sigma^{-1}(\Sigma-I)$
is invertible.  For $U\in U(M)$, $\operatorname{Tr}(H^2)$ is invariant
under conjugation, i.e., $\operatorname{Tr}\left((UHU^*)^2\right)=
\operatorname{Tr}(H^2)$, and therefore since $H\mapsto UHU^*$ preserves Lebesgue measure on ${\cal H}_M$, $UHU^*\stackrel{d}{=}H$ if and only if the corresponding densities are equal almost everywhere and everywhere by continuity, i.e., 
if and only if for every $H\in {\cal H}_M$, 
\begin{equation}\label{2.4}
{\bf d}(UHU^*)^T A\,{\bf d}(UHU^*)=
{\bf d}(H)^T A\,{\bf d}(H).
\end{equation}
Let us define 
$D_0:{\cal H}_M\longrightarrow\mathbb R^M,\ D_0(H)={\bf d}(H)$, 
and, for fixed $U\in U(M)$, define
$D_U(H):={\bf d}(UHU^*)$.  
Then \eqref{2.4} becomes
$\langle A D_UH,D_UH\rangle=
\langle A D_0H,D_0H\rangle\ \text{for every }H\in\mathcal H_M$. Next, since $A$ is symmetric, the real polarization identity applied to \eqref{2.4} gives 
$\langle A D_UH,D_UK\rangle
=\langle A D_0H,D_0K\rangle$ for every $H,K\in{\cal H}_M$.  
Equivalently,
\begin{equation}\label{2.5}
D_U^*AD_U=D_0^*AD_0.
\end{equation}
Next, we determine the kernels of these two operators. Since $D_0$ is
surjective and $A$ is invertible,
\begin{equation}\label{2.6}
\ker(D_0^*AD_0)=\ker D_0.
\end{equation}
Indeed, the inclusion from right to left is immediate. Conversely, if
$D_0^*AD_0H=0$, 
then, for every $K\in\mathcal H_M$,
$0=\langle D_0^*AD_0H,K\rangle_{\mathrm{HS}}=\langle AD_0H,D_0K\rangle$. 
Since $D_0$ is onto, $D_0K$ ranges over all of $\mathbb R^M$, and so  
$AD_0H=0$. 
The invertibility of $A$ gives $D_0H=0$.

The map $D_U$ is also surjective, since $D_U=D_0\circ\operatorname{Ad}_U,
\operatorname{Ad}_U(H)=UHU^*$, 
and $\operatorname{Ad}_U$ is an invertible linear transformation of
$\mathcal H_M$. The same arguments therefore lead to 
\begin{equation}\label{2.7}
\ker(D_U^*AD_U)=\ker D_U.
\end{equation}
Combining \eqref{2.5}--\eqref{2.7} leads to 
\begin{equation}\label{2.8}
\ker D_U=\ker D_0.
\end{equation}

Let $\mathcal D=\left\{
\operatorname{diag}(x_1,\ldots,x_M):
x_i\in\mathbb R \right\}\subset\mathcal H_M$ 
be the real subspace of diagonal Hermitian matrices. Then
$\ker D_0=\mathcal D^\perp$. 
Moreover, $\ker D_U
=\operatorname{Ad}_{U^*}(\mathcal D^\perp)$. 
Thus \eqref{2.8} implies that $\operatorname{Ad}_{U^*}(\mathcal D^\perp)
=\mathcal D^\perp$. 
Since $\operatorname{Ad}_U$ is orthogonal for the Hilbert--Schmidt
inner product, taking orthogonal complements gives
$\operatorname{Ad}_{U^*}(\mathcal D^\perp)=\mathcal D^\perp\Longrightarrow 
\operatorname{Ad}_{U^*}(\mathcal D)=\mathcal D$,
and, since $\operatorname{Ad}_{U^*}$ is its inverse,
$\operatorname{Ad}_U(\mathcal D)=\mathcal D$, i.e., 
\begin{equation}\label{2.9}
U\mathcal D U^*=\mathcal D.
\end{equation}
We now use the standard description of the normalizer of the diagonal
algebra. Let $E_{jj}=e_je_j^*, 1\leq j\leq M$. 
By \eqref{2.9},
$UE_{jj}U^*\in\mathcal D$. 
If $u_j=Ue_j$ denotes the $j$th column of $U$, then
\[
UE_{jj}U^*=u_ju_j^*.
\]
This is a rank-one orthogonal projection belonging to $\mathcal D$.
The only rank-one orthogonal projections in $\cal D$ are 
$E_{11},\ldots,E_{MM}$. 

Hence, for every $j$, there exists an index $\sigma(j)\in\{1,\dots,M\}$  and a phase $e^{i\theta_j}$ such that $u_j=e^{i\theta_j}e_{\sigma(j)}$.  
Since the columns of $U$ form an orthonormal basis, the indices
$\sigma(j)$ are distinct, and therefore $\sigma\in\mathfrak S_M$ is a permutation.
Let $P$ be the corresponding permutation matrix, defined by $Pe_j=e_{\sigma(j)}$, 
and let $D_\theta =\operatorname{diag}(e^{i\theta_1},\dots,e^{i\theta_M})$.
Then $U=PD_\theta$. Since $PD_\theta=(PD_\theta P^T)P$ 
and $PD_\theta P^T$ is again diagonal unitary, $U$ can equivalently be written
as $U=DP$ 
for some $D\in U(1)^M$.

It remains to determine which permutation matrices are allowed.
Conjugation by a diagonal unitary leaves the diagonal of a Hermitian
matrix unchanged. Under the convention $Pe_j=e_{\sigma(j)}$, one has
${\bf d}(PHP^T)=P{\bf d}(H)$, and therefore ${\bf d}(DPHP^*D^*)=P{\bf d}(H)$.
Thus, \eqref{2.4} reduces to $x^TP^TAPx=x^TAx$ 
for every $x\in\mathbb R^M$,
or equivalently
\begin{equation}\label{2.11}
P^TAP=A.
\end{equation}
Since $A=I-\Sigma^{-1}$  and $P^TP=I$, this is equivalent to
$P^T\Sigma^{-1}P=\Sigma^{-1}$.  Taking inverses gives
$P^T\Sigma P=\Sigma$. Since $P$ is orthogonal, this is equivalent to  
\begin{equation}\label{2.12}
P\Sigma P^T=\Sigma.  
\end{equation}
We have therefore proved that 
$$G_\Sigma\subseteq
\left\{DP:\ D\in U(1)^M,\ P\Sigma P^T=\Sigma\right\}.$$

Conversely, suppose that $U=DP$, where $D$ is diagonal unitary and $P$ is a permutation matrix satisfying
$P\Sigma P^T=\Sigma$.  Conjugation by $D$ fixes the diagonal entries and sends
$h_{ij}\longmapsto e^{i(\theta_i-\theta_j)}h_{ij}$. 
Since the off-diagonal entries are independent circular complex Gaussian random variables, their joint law is unchanged. Conjugation by $P$ simultaneously permutes the diagonal coordinates and the off-diagonal entries. The off-diagonal iid law is unchanged, while
$P{\bf d}(H)\sim{\cal N}(0,P\Sigma P^T)={\cal N}(0,\Sigma)$.  Thus $DPHP^*D^*\stackrel{d}{=}H$, 
proving \eqref{sym}.
\end{proof}

\begin{rem}\label{rem:symmetry-exceptional}
(i) Under the hypotheses of Proposition~\ref{prop:symmetry-group}, if $\Sigma$ has no nontrivial coordinate-permutation
symmetry, i.e., no non-identity coordinate permutation preserving it,
then $\operatorname{Aut}_{\rm perm}(\Sigma)$ is trivial and then $G_\Sigma=U(1)^M$.  

If $\Sigma=\sigma^2I$ with $\sigma^2\ne 1$, 
every permutation preserves
$\Sigma$, while $1\notin\operatorname{spec}(\Sigma)$, and therefore 
$G_{\sigma^2I}
= U(1)^M\rtimes\mathfrak S_M$.  

Finally, when $\Sigma=I$, the preceding argument in the proposition does not apply because
$A=0$. In this case, the density reduces to
$p_I(H)=C_I\exp\left\{-\frac12{\Tr}(H^2)\right\}$.  
Since ${\Tr}((UHU^*)^2)={\Tr}(H^2)$ for every
$U\in U(M)$, the law is invariant under arbitrary unitary conjugation, and as expected $G_I=U(M)$.

(ii) The requirement $1\notin\operatorname{spec}\Sigma$ 
is equivalent to the invertibility of $I-\Sigma^{-1}$ and is used in
the above proof to deduce from the invariance of the quadratic perturbation
that conjugation by $U$ preserves the diagonal subspace. If
$1\in\operatorname{spec} \Sigma$ but $\Sigma\neq I$, this argument no
longer applies, and additional symmetries are not excluded by the
proposition. The exact symmetry condition remains
$D_U^*(I-\Sigma^{-1})D_U=D_0^*(I-\Sigma^{-1})D_0$. 
Thus, the exceptional case requires a separate analysis of the null space
of $I-\Sigma^{-1}$. 
\end{rem}

Now that the eigenvalue density has been obtained and the differences from the GUE setting have been described, let us present some elements 
that lead to a discrete approximation of these eigenvalues.   
By the Cauchy interlacing theorem, 
the eigenvalues of the successive leading principal submatrices of a Hermitian matrix $H$ form a Gelfand–Tsetlin pattern (whose bottom row is the spectrum of the matrix).  In other words, if $H^{(k,k)}$, $1\le k \le M$, 
denotes the upper-left $k\times k$ principal submatrix of $H$, then the eigenvalues $\lambda_1(H^{(k,k)}) \ge \dots \ge \lambda_k(H^{(k,k)})$ rearranged in nonincreasing order are related to each other by the Cauchy interlacing inequalities
$\lambda_i(H^{(k+1,k+1)}) \ge \lambda_i(H^{(k,k)}) \ge \lambda_{i+1}(H^{(k+1,k+1)})$ for ${1 \le i \leq k < M}$, and are often arranged in a triangular array known as a Gelfand--Tsetlin (GT) pattern with bottom row $\lambda_1(H) \ge \lambda_2(H)\ge \dots \ge \lambda_M(H)$.

The next result is then \cite[Proposition 4.7]{Baryshnikov} (the set of GT patterns with fixed bottom row is a GT polytope; without fixed bottom row one speaks of the GT cone).  

\begin{prop}\label{prop:baryshnikovprop47}
Fix $\mu_1 \geq \ldots \geq \mu_M$, let $J = \diag(\mu_1, \ldots, \mu_M)$, and let $U \in U(M)$ be Haar-distributed, and form the Gelfand--Tsetlin pattern $x=(\mu_i^k)_{1 \leq i \leq k \leq M}$ from the eigenvalues of the upper-left $k\times k$ principal submatrices of $U J U^*$, $1\le k \le M$. Then, the law of $x$ is the uniform measure on the Gelfand--Tsetlin polytope, with respect to Lebesgue measure on its affine span, with bottom row $\mu = (\mu_1, \dots, \mu_M)$, i.e., conditional on the bottom row $\mu$, the induced GT pattern is uniformly distributed on the GT polytope with bottom row $\mu$.
\end{prop}

For $H\in{\cal M}_\Sigma$ with $\Sigma\ne I$, the conditional distribution given the spectrum is not in general the Haar orbital measure of Proposition~\ref{prop:baryshnikovprop47}; the required modification will instead be obtained below by a change of measure from the GUE.  Before doing so, let us state the following simple assertion which underlies the identification, for binary or ternary alphabets and iid letters, of some Brownian functionals as spectra of traceless GUE matrices.   

\begin{lemma}\label{lem:summationidentity}
Let $H= (h_{ij})_{1\le i,j\le M}\in {\cal H}_{M}$ and $x=(\mu_i^k)_{1 \leq i \leq k \leq M}$ be the Gelfand--Tsetlin pattern of the eigenvalues of 
the upper-left $k\times k$ principal submatrices of $H$, $1\le k \le M$. Then, for every $1 \leq k \leq M$,
\begin{equation*}
\sum_{i=1}^k \mu_i^k - \sum_{i=1}^{k-1} \mu_i^{k-1} = h_{kk}.
\end{equation*}
\end{lemma}
\begin{proof}

\begin{equation*}
\sum_{i=1}^k \mu_i^k - \sum_{i=1}^{k-1} \mu_i^{k-1} ={\Tr}\, H^{(k,k)} - 
{\Tr} \, H^{(k-1,k-1)} = h_{kk}, 
\end{equation*}
where the empty sums and $\operatorname{Tr}H^{(0,0)}$ are understood to be zero.  
\end{proof}

We now introduce further notation, definitions and discretizations, which will form the basic ingredients, via a discrete approximation, for obtaining our limiting shape result in the next section.  Recall that given a (deterministic) word $u \in \{1, 2, \dots, M\}^K$, $K\ge 1$, the Robinson-Schensted-Knuth (RSK) correspondence/algorithm associates to $u$ a pair of Young tableaux having the same shape encoding the increasing-subsequence statistics of the word, e.g., \cite{Fulton}.  
We use two standard forms of the RSK correspondence. When RSK is applied to a word, we refer to it as the word RSK correspondence; when it is applied to a nonnegative integer matrix (or array), we refer to it as the matrix RSK correspondence. The latter is the version used below for the geometric array 
$w=(w_{ij})$.  We read the rows in increasing order; within each row, we read the columns in increasing order, inserting the letter $j$ exactly $w_{ij}$  times. 

Let $\mathrm{GT}_M$ be the integer Gelfand--Tsetlin cone, i.e.,
the set of triangular arrays
$(x_i^k)_{1\le i\le k\le M}\in \mathbb{N}^{M(M+1)/2}$, $\mathbb N =\{0, 1, 2, \dots\}$, 
satisfying the interlacing inequalities
$x_i^k \ge x_i^{k-1} \ge x_{i+1}^k$, $1\le i<k\le M$. 
For a deterministic array
$w=(w_{ij})_{1\le i\le N,\;1\le j\le M} \in \mathbb{N}^{N\times M}$,
let $x(w)=\bigl(x_i^k(w)\bigr)_{1\le i\le k\le M} \in \mathrm{GT}_M$ 
denote its Gelfand--Tsetlin pattern under the matrix RSK
correspondence, where, for each $1\le k\le M$,
$x_i^k(w)$, $1\le i\le k$, is the $i$th row length of the
RSK shape associated with the subarray consisting of the first
$k$ columns of $w$, with zero parts appended when the shape has
fewer than $k$ nonzero rows.

Let $W^{(N)} = (W_{ij}^{(N)})_{1 \leq i \leq N, 1 \leq j \leq M}$ be an array of iid geometric weights, with parameter $q \in (0,1)$, on $\NN^{N \times M}$, and so with mean $m_q = q/(1-q)$ and variance $v_q = q/(1-q)^2$. Define $X^{(N)} = x(W^{(N)})$ and let $\PP_N$ be the law of $X^{(N)}$ on $\GT.$ Let $x(H)$ be the Gelfand--Tsetlin pattern of the eigenvalues of the successive leading principal submatrices of $H$.

To continue with our notation, for $x \in \GT$, let
$S_k(x) = \sum_{i=1}^k x_i^k$, $1 \leq k \leq M$, with $S_0(x) =0$  
and let $d_k(x) = S_k(x) - S_{k-1}(x)$, $1 \leq k \leq M$.  
Let also $D(x) = (d_1(x), \ldots, d_M(x))^T \in \ZZ^M$, 
and $\hat D_N(x) = (D(x) - m_q N \one)/\sqrt{v_qN}$.  

\smallskip

With the help of this notation, we now define the following localized probability measure on $\GT$.  

\begin{defn}\label{def:localmeasure}
Let $(L_N)_{N\ge 1}$ be a sequence of reals such that, as $N \to +\infty$, $L_N \uparrow +\infty$ and $L_N/\sqrt{N} \to 0$.  Let $N_0 \geq M$ be a constant such that for every $N \geq N_0$, $[m_qN - L_N \sqrt{v_qN}, m_qN + L_N \sqrt{v_qN}]$ contains a nonnegative integer, and let $E_N = \left \{ x \in \GT: \max_{1 \leq i \leq M} \abs{\frac{x_i^M - m_q N}{\sqrt{v_qN}}} \leq L_N \right \}$.  For $N \geq N_0$ and $\Sigma$ positive definite, define $\Ploc$ on $\GT$ via 
\begin{equation*}
\frac{d\Ploc}{d \PP_N}(x) = \frac{1}{\Zloc} \exp \left ( \frac{1}{2} \hat D^T_N(x) ( I - \Sigma^{-1}) \hat D_N(x) \right ) {\bf 1}_{E_N}(x), 
\end{equation*} 
where 
\begin{equation*}
\Zloc = \EE \left [ \exp \left ( \frac{1}{2} \hat D^T_N(X^{(N)}) ( I -\Sigma^{-1}) \hat D_N(X^{(N)})\right ) {\bf 1}_{E_N}(X^{(N)})\right].
\end{equation*}
\end{defn}

The following proposition shows that the localized change of measure is well defined by verifying that its normalizing constant (partition function) is strictly positive and finite.

\begin{prop}
$0< \Zloc < +\infty$ for every $N \geq N_0.$
\end{prop}
\begin{proof}
By interlacing and the definition of $E_N$, for every $1 \leq i \leq k \leq M$, 
\begin{equation*}
0 \leq x_M^M \leq x_i^k \leq x_1^M \leq \abs{x_1^M -m_qN} + \abs{m_qN} \leq L_N \sqrt{v_qN} + \abs{m_qN}.
\end{equation*} 
Hence $E_N$ is finite. Since the exponential weight is finite on $E_N$, it follows that $Z_N^{\Sigma,\mathrm{loc}}<\infty$.

Next, choose a nonnegative integer
$a\in[m_qN-L_N\sqrt{v_qN},\,m_qN+L_N\sqrt{v_qN}]$.
Since $N\ge M$, consider the deterministic $N\times M$ array
$w^{(a)}=(w_{ij}^{(a)})$ defined by $w_{ij}^{(a)} 
=\begin{cases}
a,&1\le i\le M,\ j=M+1-i,\\
0,&\text{otherwise}.
\end{cases}
$
The insertion word associated with $w^{(a)}$ by the matrix RSK correspondence is
$M^a(M-1)^a\cdots1^a$, under our chosen matrix RSK ordering convention, 
and successive row insertions show that its shape is the rectangular partition
$(a^M)=(a,\ldots,a)$.
Hence the bottom row of $x(w^{(a)})$ satisfies
$x^M(w^{(a)})=(a,\ldots,a)$, and therefore $x(w^{(a)})\in E_N$. Since the geometric distribution assigns positive mass to every nonnegative integer and the entries of $W^{(N)}$ are independent,
$\mathbb P\{W^{(N)}=w^{(a)}\}>0$.
It follows that $\mathbb P_N(E_N)>0$.  Since the exponential weight in the definition of $Z_N^{\Sigma,\mathrm{loc}}$ is positive, it follows that 
$Z_N^{\Sigma,\mathrm{loc}}>0$.
\end{proof}

The next result identifies the diagonal statistic as a linear image of the rescaled Gelfand--Tsetlin pattern and relates it directly to the column sums of the underlying array.  Below, let $\xi_N(x) = 
\left ( {(x_i^k-m_qN)}/{\sqrt{v_qN}} \right )_{1 \leq i \leq k \leq M} \in \RR^{M(M+1)/2}$.

\begin{prop}\label{prop:column-sum}
Let $N\ge 1$, let $w\in\mathbb N^{N\times M}$, and let
$x=x(w)$ be its matrix RSK Gelfand--Tsetlin pattern.
Define the linear map $T:\mathbb R^{M(M+1)/2}\to\mathbb R^M$ by
$T_k(y)=\sum_{i=1}^k y_i^k-\sum_{i=1}^{k-1}y_i^{k-1}$, $1\le k\le M$.  
Then, $T(\xi_N(x)) = \hat D_N(x)$ and for every $1 \leq k \leq M$,
\begin{equation*}
d_k(x) = \sum_{a=1}^N w_{ak}.
\end{equation*}
\end{prop}
\begin{proof}
Let $P_k$ be the insertion tableau from the matrix RSK correspondence applied to the sub-array of the first $k$ columns of $w$. The number of boxes of $P_k$ is
\begin{equation*}
\sum_{a=1}^N \sum_{b=1}^k w_{ab} = \sum_{i=1}^k x_i^k,
\end{equation*}
where the left-hand side is deduced because $P_k$ contains one box for each unit of total weight in the sub-array by construction, and the right-hand side comes directly from the shape of $P_k$. Then, 
\begin{equation*}
d_k(x) = \sum_{i=1}^k x_i^k - \sum_{i=1}^{k-1} x_i^{k-1}= \sum_{a=1}^N \sum_{b=1}^k w_{ab} - \sum_{a=1}^N \sum_{b=1}^{k-1} w_{ab}= \sum_{a=1}^N w_{ak}.
\end{equation*}
Next, observe that for $1 \leq k \leq M$,
\begin{align*}
T_k(\xi_N(x)) &= \sum_{i=1}^k \frac{x_i^k- m_qN}{\sqrt{v_qN}} - \sum_{i=1}^{k-1} \frac{x_i^{k-1}-m_qN}{\sqrt{v_qN}}= \frac{d_k(x)-m_qN}{\sqrt{v_qN}}.
\end{align*}
Thus, $T(\xi_N(x)) = \hat D_N(x)$ as desired.
\end{proof}

\section{Convergence of Discrete RSK Models to Gelfand--Tsetlin Patterns}

The purpose of this section is to establish the convergence of the discrete RSK models introduced above to the Gelfand–Tsetlin patterns of the Gaussian matrix ensemble ${\cal M}_\Sigma$. The argument is based on a change of measure from the GUE case; a similar change of measure approach was used in \cite{dieker}. Indeed, the exponential tilt defining $\Ploc$ is the discrete counterpart of the Radon–Nikodym factor relating ${\cal M}_\Sigma$ to the GUE. After establishing the uniform integrability needed to pass this tilt through the known GUE/RSK limit, we obtain the desired convergence for positive definite $\Sigma$. We then extend the result to positive semidefinite covariance matrices by approximation.

Throughout this section, unless explicitly stated otherwise, $\Sigma$ is positive definite, and whenever we consider $E_N$, $\Ploc$, or $\Zloc$, then $N \geq N_0$. Additionally, for $y \in \RR^{M(M+1)/2}$, define   
\begin{equation*}
    \Psi(y) = \Psi_\Sigma(y) = \exp \left ( \frac{1}{2} T(y)^T(I - \Sigma^{-1}) T(y) \right ).
\end{equation*}

To pass the discrete change of measure through the GUE/RSK scaling limit, we first establish a uniform $L^p, p>1$ bound for its unnormalized exponential weight, thus providing a uniform integrability result.  

\begin{thm}\label{thm:uniform}
There exists $p>1$ such that
\begin{equation*}
\sup_{N \geq N_0} \EE_{\PP_N} \left ( \Psi(\xi_N(X^{(N)}))^p \one_{E_N}(X^{(N)})\right ) < \infty.
\end{equation*}
\end{thm}
\begin{proof}
Let $A:= I - \Sigma^{-1}$.  Since $\Sigma$ is positive definite, $\lambda_{\textrm{max}}(A) < 1$, and let $a := \max \left ( 0, \lambda_{\textrm{max}}(A) \right)$. If $a=0$, the result follows trivially as then $\Psi(y) \le 1$, for all $y$.
So, let $a>0$ and choose $1< p < 1/a$ and $\eta>0$ such that $(1 + \eta) p a < 1$ and set $c = pa/2$. For every $z\in {\mathbb R}^M,
z^TAz\le\lambda_{\max}(A)\|z\|^2\le a\|z\|^2$.  Thus, 
$$\Psi(y)^p = \exp\left\{\frac p2 T(y)^TAT(y)\right\}
\le
\exp\left\{
\frac{pa}{2}\|T(y)\|^2
\right\}
=\exp\{c\|T(y)\|^2\}, $$
and so 
${\Psi(y)^p\le
\exp\left(c\sum_{k=1}^M T_k(y)^2\right)}$. 

For $1 \leq k \leq M$ and $x \in E_N$, 
$\abs{({x_i^M-m_qN})/{\sqrt{v_qN}}}\le L_N$, and interlacing gives $x_M^M\le x_i^k\le x_1^M$, so actually
$|({x_i^k-m_qN})/{\sqrt{v_qN}}|\le L_N$, $1\le i\le k\le M$.
Therefore, 
$$ 
|T_k(\xi_N(x))| =
\left|
\sum_{i=1}^k
\frac{x_i^k-m_qN}{\sqrt{v_qN}}
-
\sum_{i=1}^{k-1}
\frac{x_i^{k-1}-m_qN}{\sqrt{v_qN}}
\right|\le 
(2k-1)L_N,$$ 
so $\abs{T_k ( \xi_N(x))} \le (2M-1) L_N$, $1 \le k \le M$.  

By Proposition~\ref{prop:column-sum}, under $\PP_N$, 
$T_k(\xi_N(X^{(N)}))=S_{N,k}:=N^{-1/2}\sum_{r=1}^N Y_r^{(k)}$,
where, for each $k$, the $Y_r^{(k)}$ are iid standardized geometric random variables, and the random variables $S_{N,1},\dots,S_{N,M}$ are independent.
(For different $k$, these variables are independent since they depend on different columns of the iid geometric array.) For convenience, write $Y_r:=Y_r^{(1)}$ and $S_N:=N^{-1/2}\sum_{r=1}^N Y_r$. The variables $S_{N,k}, 1\leq k\leq M$, are independent copies of $S_N$.  

Thus, since ${\bf 1}_{E_N}
\le \prod_{k=1}^M {\bf 1}_{\{|T_k|\le (2M-1)L_N\}}$, 
\begin{equation*}
\EE_{\PP_N} \left [ \Psi ( \xi_N(X^{(N)}))^p \one_{E_N}(X^{(N)}) \right] \leq \left ( \EE \left [ \exp (cS_N^2) \one_{\{\abs{S_N} \leq (2M-1)L_N\}} \right ] \right)^M.
\end{equation*}

Let $K(\theta) = \log \EE [\exp(\theta Y_1)]$. Since $K(\theta) = \theta^2/2 + O(|\theta|^3)$, for $|\theta|$ in a neighborhood of $0$, for the previously chosen $\eta > 0$, there exists $\delta > 0$ such that $K(\theta) \leq (1+ \eta)\theta^2/2$ for $\abs{\theta} \leq \delta$.  

Next, the Hubbard--Stratonovich identity
$e^{cs^2}
=\frac1{\sqrt{2\pi}}\int_{\mathbb R}
\exp\left(-\frac{u^2}{2}+\sqrt{2c}\,us\right)\,du$ leads to 
$\mathbb E\left[
e^{cS_N^2}\mathbf1_{\{|S_N|\le (2M-1)L_N\}}
\right]=
\frac1{\sqrt{2\pi}}
\int_{\mathbb R}e^{-u^2/2}
\mathbb E
\left[
e^{\sqrt{2c}\,uS_N}
\mathbf1_{\{|S_N|\le (2M-1)L_N\}}
\right]du$. 
Therefore, 
\begin{align*}
\EE \left [ \exp (cS_N^2) \one_{\{\abs{S_N} \leq (2M-1)L_N\}} \right ] &\leq \frac{1}{\sqrt{2 \pi}} \int_{\abs{u} \leq \frac{\delta {\sqrt N}}{\sqrt{2c}}}\exp \left  (-u^2/2+ NK\left (u\sqrt{2c/N}\right ) \right ) du\\ 
&+  \frac{1}{\sqrt{2 \pi}} \int_{\abs{u} > \frac{\delta {\sqrt N}}{\sqrt{2c}}}\exp \left (-u^2/2+ \abs{u}(2M-1)L_N\sqrt{2c} \right ) du. 
\end{align*}
We first estimate the first integral on the right-hand side. 
For $|u|\le\delta\sqrt{{N}/{2c}}$, dropping the indicator,
\begin{equation*}
\mathbb E e^{\sqrt{2c}uS_N}
=\mathbb E
\exp\left(
u\sqrt{\frac{2c}{N}}\sum_{r=1}^NY_r
\right)=\left(
\mathbb E
e^{u\sqrt{2c/N}Y_1}
\right)^N=
\exp\left\{
NK\left(u\sqrt{\frac{2c}{N}}\right)
\right\}. 
\end{equation*}
Therefore, $$-\frac{u^2}{2}+NK\left(u\sqrt{\frac{2c}{N}}\right)
\le -\frac{u^2}{2}+N\frac{1+\eta}{2}\frac{2cu^2}{N}
=-\frac12\left(
1-(1+\eta)pa \right)u^2,
$$
since $2c=pa$.  By construction,  $1-(1+\eta)pa>0$ 
and so the first integral is at most
$$\frac1{\sqrt{2\pi}}
\int_{\mathbb R}
e^{-\frac12[1-(1+\eta)pa]u^2}\,du
=
\frac1{\sqrt{1-(1+\eta)pa}}.$$

Now for the tail integral, on the event $|S_N|\le (2M-1)L_N$, it holds that 
$e^{\sqrt{2c}uS_N}
\le e^{\sqrt{2c}|u|(2M-1)L_N}$, and thus the tail contribution is bounded by
$$\frac1{\sqrt{2\pi}}
\int_{|u|>\frac{\delta \sqrt N}{\sqrt{2c}}}
\exp\left\{
-\frac{u^2}{2}
+\sqrt{2c}\,(2M-1)L_N|u|
\right\}du.$$
Since 
${L_N}/{\sqrt N}\to 0$, there exists $N_1$ such that, for $N\ge N_1$,
$8c(2M-1)L_N\le\delta\sqrt N$.  Hence for $N\ge N_1$ and $|u|>\delta\sqrt{N/2c}$,
$\sqrt{2c}\,(2M-1)L_N|u|
\le {u^2}/{4}$. 
Therefore
$-{u^2}/{2}
+\sqrt{2c}\,(2M-1)L_N|u|
\le-{u^2}/{4}$,
so
$$
\frac1{\sqrt{2\pi}}
\int_{|u|>\frac{\delta \sqrt N}{\sqrt{2c}}}
e^{-u^2/2+\sqrt{2c}(2M-1)L_N|u|}\,du
\le
\frac1{\sqrt{2\pi}}
\int_{\mathbb R}e^{-u^2/4}\,du
=\sqrt2.$$ 
Combining both estimates gives  
$$\EE \left [ \exp (cS_N^2) \one_{\{\abs{S_N} \leq (2M-1)L_N\}} \right ] \le 
\frac{1}{\sqrt{1 - (1+ \eta) pa}} + \sqrt{2}, 
$$
for all $N\ge N_1$.  
For the finitely many $N_0 \le N < N_1$, on $|S_N|\le (2M-1)L_N$,
$e^{cS_N^2}\le e^{c(2M-1)^2L_N^2}$, so the corresponding expectation is at most 
$\exp ( c(2M-1)^2 L_N^2)$, leading to a uniform bound for all $N \geq N_0$.
\end{proof}

To identify the limit of the discrete exponential tilt, we express the Gelfand--Tsetlin law of the $\Sigma$-ensemble as a change of measure from the GUE case.

\begin{prop}\label{prop:changearraymeasure}
Let $H_0 \in {\cal M}_I$ and let $H \in {\cal M}_\Sigma$. Then, for every bounded measurable $F: \RR^{M(M+1)/2} \to \CC$,
\begin{equation*}
\EE_\Sigma[F(x(H))] = \frac{\EE_I[F(x(H_0))\Psi(x(H_0))]}{\sqrt{\det \Sigma}}
\end{equation*}
and $\EE_I[\Psi(x(H_0))] = \sqrt{\det \Sigma}.$
\end{prop}
\begin{proof}
Letting $p_\Sigma$ and $p_{I}$ be the respective matrix densities of $H$ and $H_0$, given 
in the proof of Proposition~\ref{prop:JPDF}, observe that
\begin{equation*}
\frac{p_\Sigma(H_0)}{p_{I}(H_0)} = \frac{1}{\sqrt{\det \Sigma}} \exp 
\left ( \frac{1}{2} \mathbf d(H_0)^T (I-\Sigma^{-1}) \mathbf d (H_0)\right ).
\end{equation*}
Therefore, using $T(x(H_0)) = \mathbf d(H_0)$ given in Lemma~\ref{lem:summationidentity}, and for emphasis, denoting by 
${\mathbb E}_I$ and $\EE_\Sigma$ the expectations with respect to the respective laws of $H_0$ and $H$, 
\begin{align*}
\EE_\Sigma [F(x(H))] &= \EE_I \left [ \frac{F(x(H_0))}{\sqrt{\det \Sigma}} \exp \left ( \frac{1}{2} \mathbf d(H_0)^T (I-\Sigma^{-1}) \mathbf d (H_0)\right ) \right]\\
&= \frac{1}{\sqrt{\det \Sigma}} \EE_I\left [ F(x(H_0)) \Psi(x(H_0))\right ].
\end{align*}
Letting $F=1$ yields $\EE_I[\Psi(x(H_0))] = \sqrt{\det \Sigma}.$
\end{proof}

Combining the uniform $L^p$--bound with the Gaussian change-of-measure identity yields the weak convergence of the localized discrete RSK model to the Gelfand–Tsetlin pattern of the $\Sigma$-ensemble.

\begin{thm}\label{thm:lemma21gen}
Let $H \in \mathcal M_\Sigma.$ Then, under $\Ploc$, as $N\to +\infty$, $\xi_N(X^{(N)})\Longrightarrow x(H)$ in $\RR^{M(M+1)/2}.$
\end{thm}
\begin{proof}
Let $H_0 \in \mathcal M_I$ and recall that \cite[Lemma 2.1]{BGH} asserts 
\begin{equation*}
\xi_N(X^{(N)}) \Rightarrow x(H_0), 
\end{equation*}
and therefore the family of laws of $\xi_N(X^{(N)}), N\ge N_0$, is tight.  
Since the bottom-row projection is continuous, the corresponding family of bottom-row laws is also tight.  Therefore, fixing $\varepsilon >0$, we may choose $R< \infty$ such that
\begin{equation*}
\sup_{N \geq N_0} \PP_N \left ( \max_{1 \leq i \leq M} \abs{(\xi_N(X^{(N)}))_i^M}> R \right ) < \varepsilon.
\end{equation*}
Next, choose an integer $N_\varepsilon \geq N_0$ such that $L_N \geq R$ for every $N \geq N_\varepsilon.$ By Definition \ref{def:localmeasure}, $\PP_N(X^{(N)} \not \in E_N) < \varepsilon$ for $N \geq N_{\varepsilon}$, so $\one_{E_N} (X^{(N)}) \stackrel{N\to +\infty}\longrightarrow 1$ in probability under $\PP_N.$

For a bounded continuous $F: \RR^{M(M+1)/2} \to \RR$, the continuity of $F$ and of $\Psi$, as well as Slutsky's theorem using $\one_{E_N} (X^{(N)}) \to 1$ in probability under $\PP_N$, lead to 
\begin{equation*}
F(\xi_N(X^{(N)}))\Psi(\xi_N(X^{(N)}))\one_{E_N}(X^{(N)}) \Longrightarrow F(x(H_0)) \Psi(x(H_0)).
\end{equation*}
For bounded $F$, Theorem~\ref{thm:uniform} gives an $L^p$-bound for some $p>1$, since 
$
|F(\xi_N)\Psi(\xi_N){\bf 1}_{E_N}|^p
\le \|F\|_\infty^p \Psi(\xi_N)^p{\bf 1}_{E_N}$,
which in turn gives uniform integrability. 
Therefore,
\begin{equation*}
\EE_{\PP_N} \left [ F(\xi_N(X^{(N)})) \Psi(\xi_N(X^{(N)})) \one_{E_N} (X^{(N)}) \right] \longrightarrow \EE_I [F(x(H_0)) \Psi(x(H_0))].
\end{equation*}
Taking $F=1$ and using Proposition~\ref{prop:changearraymeasure},
\begin{equation*}
\Zloc \longrightarrow \EE_I [\Psi(x(H_0))] = \sqrt{\det \Sigma}.
\end{equation*}
Therefore,
\begin{align*}
\EE_{\Ploc} [F(\xi_N(X^{(N)}))] & = \frac{1}{\Zloc} \EE_{\PP_N} \left [ F(\xi_N(X^{(N)})) \Psi(\xi_N(X^{(N)})) \one_{E_N} (X^{(N)}) \right]\\
&\longrightarrow \frac{1}{\sqrt{\det \Sigma}} \EE_I [F(x(H_0)) \Psi(x(H_0))] 
= \EE_\Sigma [F(x(H))],
\end{align*}
where the last equality follows from Proposition~\ref{prop:changearraymeasure}.
\end{proof}

We now remove the standing assumption of positive definiteness and treat positive semidefinite covariance matrices via an approximation scheme.  
We extend the convergence to possibly singular covariance matrices by regularizing $\Sigma$ on its kernel and letting the regularization vanish as $N\to +\infty$.

Below, $\mathbb{P}_N^{\Sigma_N, \mathrm{loc}}$ is the local measure given in Definition \ref{def:localmeasure} but taking  $\Sigma_N$ instead of $\Sigma$.   

\begin{thm}\label{thm:lemma21singgen}
Let $H \in \mathcal M_\Sigma$ and let $\Sigma$ be positive semidefinite. Then, there exists a deterministic non-increasing sequence $\varepsilon_N>0$ with $\varepsilon_N\to 0$
such that, as $N \to +\infty$,
\begin{equation*}
\xi_N(X^{(N)} ) \Longrightarrow x(H)
\end{equation*}
under $\mathbb{P}_N^{\Sigma_N, \mathrm{loc}}$, where $\Sigma_N = \Sigma+ \varepsilon_N P^{\perp}$ and $P^\perp$ is the orthogonal projection onto the kernel of $\Sigma$.  
\end{thm}
\begin{proof}

Let $P=P_{\operatorname{Ran}\Sigma}$ be the orthogonal projection onto the range of $\Sigma$, while $P^\perp=I-P=P_{\ker\Sigma}$ is the orthogonal projection onto its kernel. 
Fix $\varepsilon > 0$ and let $\Sigma_\varepsilon = \Sigma+\varepsilon P^{\perp}$. ($\Sigma_\varepsilon$ is positive definite from the very definitions of $P$ and $P^\perp$.)  Let the law of $\xi_N(X^{(N)})$ under $\mathbb{P}_N^{\Sigma_\varepsilon, \mathrm{loc}}$ be denoted by $\nu_N^{(\varepsilon)}$ and the law of $x(H_\varepsilon)$ for $H_\varepsilon \in \mathcal M_{\Sigma_\varepsilon}$ be denoted by $\nu^{(\varepsilon)}$. Since $\Sigma_\varepsilon$ is positive definite, Theorem~\ref{thm:lemma21gen} implies that $\nu_N^{(\varepsilon)} \Rightarrow \nu^{(\varepsilon)}$ as $N \to \infty$.

Since $\mathbf d(H_\varepsilon)$ has law $N(0, \Sigma_\varepsilon)$ and $\Sigma_\varepsilon \to \Sigma$, it follows that $\mathbf d(H_\varepsilon) \Rightarrow \mathbf d(H)$. Moreover, the off-diagonal law is independent of $\varepsilon$, so $H_\varepsilon \Rightarrow H$ in $\mathcal H_M$. The ordered eigenvalues of Hermitian matrices are Lipschitz continuous \cite[Corollary A.6]{Anderson-Guionnet-Zeitouni}, and the same holds when restricting to each principal minor. Thus, $H \mapsto x(H)$ is continuous, and by the continuous mapping theorem, $\nu^{(\varepsilon)} \Rightarrow \nu$, where $\nu$ is the law of $x(H)$.

Let $\rho$ be the Prokhorov metric and choose a monotonically non-increasing sequence of positive reals $c_k \to 0$ such that $\rho(\nu^{(c_k)}, \nu) \le 1/(2k)$. For each $k$, choose $N_k \geq N_0$ such that $\rho(\nu_N^{(c_k)}, \nu^{(c_k)}) \leq 1/(2k)$ for $N \geq N_k$ and if $k \geq 2$, additionally require $N_k > N_{k-1}$. For $N< N_1$, let $\varepsilon_N = c_1$, and for $N_k \leq N < N_{k+1}$, let $\varepsilon_N = c_k$. By the triangle inequality, for $N_k \leq N < N_{k+1}$
\begin{equation*}
\rho(\nu_N^{(\varepsilon_N)}, \nu) \leq \rho(\nu_N^{(c_k)}, \nu^{(c_k)}) + \rho(\nu^{(c_k)}, \nu) \le \frac{1}{k},
\end{equation*}
and by construction, $\varepsilon_N\to 0$. Therefore, $\nu_N^{(\varepsilon_N)} \Rightarrow \nu$.
\end{proof}

\section{Markov Random Words and Representation Rigidity}

We now apply the preceding random-matrix framework to the limiting Brownian functionals arising from Markov random words. We first identify the covariance structure for which the projected Brownian motion is isotropic on the trace-zero hyperplane $\mathbf 1^\perp$, namely, 
$$Q\Sigma Q=cQ,
\qquad
Q=I-\frac1M\one \one^T.$$

Under this condition, the corresponding limiting functional admits a Gaussian random-matrix representation. We then show that, within the natural class of Hermitian Gaussian matrices with independent Gaussian off-diagonal entries, this condition is in fact rigid: a conditional maximal-eigenvalue representation valid for every endpoint $d\in\mathbb R^M$ forces both $Q\Sigma Q=cQ$ and equality of all off-diagonal variances. Finally, for Markov chains with uniform stationary distribution, we characterize explicitly the transition matrices whose asymptotic covariance has the form $\Sigma=c\left(I-\frac1M\one \one^T\right)$. 

Let $\Sigma$ be an $M\times M$ positive semidefinite matrix, with $M\ge2$, let $B^\Sigma$ be Brownian motion in ${\mathbb R}^M$ with covariance matrix $\Sigma$, and define the last-passage functional 
\begin{equation}\label{LastPassFunct}
F_\Sigma
:=\max_{0=t_0\le t_1\le\cdots\le t_M=1}
\sum_{i=1}^M\left(B_i^\Sigma(t_i)-B_i^\Sigma(t_{i-1})\right).
\end{equation}
Recall from \cite{HL2} that, for a finite-state Markov chain with uniform stationary distribution and asymptotic covariance matrix $\Sigma$, $F_\Sigma$ is the limiting random variable for the centered and normalized length of the longest weakly increasing subsequence.

Throughout,  let $Q = I - \frac{1}{M} \one \one^T$ be the orthogonal projection onto $\one^\perp$.  The following proposition decomposes $\Sigma$ along $\one^\perp\oplus\operatorname{span}{\one}$  under the condition $Q\Sigma Q=cQ$ and gives the corresponding criterion for positive semidefiniteness.

\begin{prop}\label{prop:blockform}
Let $c \ge 0$ and let $\Sigma$ be a symmetric $M\times M$ matrix such that $Q \Sigma Q = cQ$. Then,
\begin{equation*}
\Sigma = cQ + \frac{1}{M} \left ( Q \Sigma \one \one^T + \one \one ^T \Sigma Q \right ) + \frac{1}{M^2} \left ( \one^T \Sigma \one \right ) \one \one ^T.
\end{equation*}
If $c>0$, then $\Sigma$ is positive semidefinite if and only if  $c \one^T \Sigma \one \geq \norm{Q \Sigma \one}^2$.  If $c=0$, then $\Sigma$ is positive semidefinite if and only if
$\one^T \Sigma \one \ge 0$ and $\|Q \Sigma \one\|^2 =0$.  
\end{prop}
\begin{proof}
Expanding $\Sigma = I \Sigma I$ with $I = Q+ \frac{1}{M} \one \one^T$ gives
\begin{align*}
        \Sigma &= \left ( Q + \frac{1}{M} \one \one ^T \right ) \Sigma \left ( Q + \frac{1}{M} \one \one^T \right ) \\
        &= Q \Sigma Q + \frac{1}{M} \left ( Q \Sigma \one \one^T + \one \one^T \Sigma Q \right )+\frac{1}{M^2} (\one^T \Sigma \one) \one \one^T\\
        &= cQ + \frac{1}{M} \left ( Q \Sigma \one \one^T + \one \one ^T \Sigma Q \right ) + \frac{1}{M^2} \left ( \one^T \Sigma \one \right ) \one \one ^T.
\end{align*}

Next, let $c>0$.   With respect to the orthogonal decomposition $\RR^M = \one^\perp \oplus \textrm{span} \{ {\one}/{\sqrt M}\}$, we have 
$$\Sigma=
\begin{pmatrix}
cI_{\one^\perp}
&
M^{-1/2}Q\Sigma\one\\[2mm]
M^{-1/2}(Q\Sigma\mathbf1)^T
&
M^{-1}\one^T\Sigma\one
\end{pmatrix}.$$
Then, the Schur complement of the $\one^\perp$ block is 
$\frac1M\left(\one^T\Sigma\one -\frac1c\|Q\Sigma\one \|^2\right)$.
By the Schur complement criterion, $\Sigma$ is positive semidefinite if and only if $c \one ^T \Sigma \one \geq \norm{Q \Sigma \one}^2.$

If $c=0$, the stated conditions imply $Q\Sigma\one=0$. Substituting into the decomposition yields $\Sigma=(\mathbf1^T\Sigma\one)
\one\one^T/M^2$, 
which is positive semidefinite precisely when $\one^T\Sigma\one \ge 0$.  
Conversely, if $\Sigma$ is positive semidefinite and $Q\Sigma Q=0$, then $\Sigma^{1/2}Q=0$, hence $\Sigma Q=0$, so $Q\Sigma\one =0$, while $\one ^T\Sigma\one \ge 0$.
\end{proof}

The next theorem identifies the joint law of the Brownian endpoint and its last-passage functional under projected isotropy, while the corollary gives the corresponding identity for a prescribed endpoint.

\begin{thm}\label{thm:qsigmaq}
    Let $B$ be a standard Brownian motion in $\mathbb R^M$ and let $c \geq 0$. Then, $QB^\Sigma$ and $\sqrt{c} QB$ have the same law if and only if
    \begin{equation*}
        Q \Sigma Q = cQ.
    \end{equation*}
    Moreover, if $\Sigma = cQ$, then for $H_0 \in {\cal M}_I$, 
 $(B ^\Sigma(1), F_\Sigma)$ has the same law as
 \begin{equation*}
     \left (\sqrt c Q \mathbf d(H_0), \lambda_{\textrm{max}} (\sqrt c (H_0 - \frac{\Tr H_0}{M} I ))\right ).
 \end{equation*}
 Additionally, if $Q \Sigma Q = cQ$, there exists a centered Gaussian random variable $Z$ that is jointly Gaussian with $Q \mathbf d(H_0)$ and independent of the off-diagonal entries of $H_0$ and of $\Tr H_0$ such that
 \begin{equation*}
     (B^\Sigma(1), F_\Sigma) \stackrel{d} = \left ( \sqrt{c} Q \mathbf d(H_0) + Z \one, \lambda_{\textrm{max}} \left ( \sqrt{c} \left ( H_0- \frac{ \Tr H_0}{M}I \right ) + ZI \right ) \right ).
 \end{equation*}
\end{thm}
\begin{proof}
    Both processes are centered Gaussian with covariance functions $(s \wedge t) Q \Sigma Q$ and $c(s \wedge t)Q$ respectively, so they have the same law if and only if $Q \Sigma Q = cQ$. If $\Sigma =cQ$, then $B^\Sigma = Q B^\Sigma$ almost surely and thus $B^\Sigma$ has the same law as $\sqrt c QB$.

Next, taking $B^I=B$ and $k_1 =M$ in \cite[equation (3.2)]{BGH} shows that $(B(1), F_I)$ has the same law as $(\mathbf d(H_0), \lambda_{\max}(H_0)).$  Since $QB = B-(\one^TB/M) \one$, the increments of $\one^TB/M$ telescope, and
\begin{equation*}
        \max_{0=t_0 \leq t_1 \leq \cdots \leq t_M =1} \sum_{i=1}^M \sqrt{c} \left ((QB)_i(t_i)-(QB)_i(t_{i-1}) \right) = \sqrt{c} \left (F_I - \frac{\one^T B(1)}{M} \right ).
\end{equation*}

Applying this centering to the preceding joint identity and using $\one^T \mathbf d(H_0) = \Tr H_0$ shows
    \begin{align*}
        \left (B^\Sigma(1), F_\Sigma \right )&  \stackrel{d}= \left ( \sqrt{c} Q \mathbf d(H_0), \sqrt{c} \left ( \lambda_{\textrm{max}}(H_0) - \frac{\Tr H_0}{M} \right ) \right )\\
        & = \left ( \sqrt{c} Q \mathbf d(H_0),\lambda_{\textrm{max}} \left ( \sqrt{c} \left ( H_0- \frac{\Tr H_0}{M}I \right ) \right ) \right ).
    \end{align*}

Finally, assume that $Q\Sigma Q=cQ$, with $c>0$,  and set
$X:=\sqrt c\,Q \mathbf d(H_0)$.
Then $X$ is a centered Gaussian vector supported on $\one^\perp$, with covariance $cQ$. By Proposition~\ref{prop:blockform},
$\Sigma = cQ+\frac1M\left(Q\Sigma\mathbf1\mathbf1^T+
\one\one^T\Sigma Q
\right)
+\frac1{M^2}
(\mathbf1^T\Sigma\mathbf1)\mathbf1\mathbf1^T$,
and, since $\Sigma$ is positive semidefinite,
$c\,\one^T\Sigma\one \ge
\|Q\Sigma\mathbf1\|^2$.
Let  
$\xi=(\xi(t))_{0\leq t\leq1}$ be a centered scalar Brownian motion,
independent of $H_0$, with variance parameter $\sigma_\xi^2
:=\frac1{M^2}\left(
\mathbf1^T\Sigma\mathbf1
-\frac1c\|Q\Sigma\one\|^2\right)$, 
and define
$$Z:=\frac1{Mc}(Q\Sigma\mathbf1)^T X+\xi(1)
=\frac1{M\sqrt c}(Q\Sigma\mathbf1)^TQ \mathbf d(H_0)+\xi(1).$$
Then, 
$\operatorname{Cov}(X,Z)=
\frac1M Q\Sigma\one$
and
$\operatorname{Var}(Z)
=\frac1{M^2}\one^T\Sigma\one$.
Therefore, Proposition~\ref{prop:blockform} gives
$\operatorname{Cov}(X+Z\mathbf1)=\Sigma$,
and thus 
$X+Z\mathbf1\stackrel d=B^\Sigma(1)$.
Moreover, $Z$ is jointly Gaussian with $Q \mathbf d(H_0)$, while it is independent of the off-diagonal entries of $H_0$ and of $\operatorname{Tr}H_0$.

To realize this coupling at the process level, let
$\beta=(\beta(t))_{0\leq t\le1}$ be a standard
$\mathbb R^M$-valued Brownian bridge, independent of $H_0$ and of
$\xi$, and define
$
B(t):=t\,\mathbf d(H_0)+\beta(t)$, $0\le t\le 1$.
Since $\mathbf d(H_0)\sim N(0,I)$  and is independent of $\beta$,
$B$ is a standard Brownian motion in $\mathbb R^M$, with
$B(1)=\mathbf d(H_0)$ 
almost surely. Set
$
X(t):=\sqrt c\,QB(t)$, $0\le t\le 1$. 
Then, $X(1)=\sqrt c\,Q\mathbf d(H_0)=X$. 
Define $A(t):=\frac1{Mc}(Q\Sigma\mathbf1)^TX(t)+\xi(t)$, $0\leq t\le 1$.
For $0\le s,t\le 1$, $\operatorname{Cov}(X(s),X(t))=(s\wedge t)cQ$, 
$\operatorname{Cov}(X(s),A(t))
=(s\wedge t)\frac1M Q\Sigma\one$, 
and
$\operatorname{Cov}(A(s),A(t))
=(s\wedge t)\frac1{M^2}\mathbf1^T\Sigma\one$. 
Hence the centered Gaussian process
$\widetilde B(t):=X(t)+A(t)\one$
satisfies $\operatorname{Cov}(\widetilde B(s),\widetilde B(t))
=(s\wedge t)\Sigma$, 
and therefore 
$\widetilde B\stackrel d=B^\Sigma$. 
Moreover, since $X(1)=X$,
$A(1)
=\frac1{Mc}(Q\Sigma\mathbf1)^TX+\xi(1)
=Z$.  

For every admissible 
$0=t_0\leq t_1\leq\cdots\leq t_M=1$, 
the common scalar process telescopes: 
$\sum_{i=1}^M
\bigl(A(t_i)-A(t_{i-1})\bigr)
=A(1)-A(0)
=A(1)$. Therefore, $F_{\widetilde B}=F_X+A(1)=F_X+Z$, 
pathwise.  Finally, 

$$(X(1),F_X,\xi(1))
\stackrel d=
\left(
\sqrt c\,Q\mathbf d(H_0),
\lambda_{\max}\!\left(
\sqrt c\left(H_0-\frac{\operatorname{Tr}H_0}{M}I\right)
\right),
\xi(1)
\right),$$
by the preceding joint identity and the independence of $\xi(1)$ from $(X(1),F_X)$ and from $H_0$.  Applying
$$(x,f,u)\longmapsto
\left(
x+z(x,u)\one,\,f+z(x,u)
\right),$$ where $z(x,u):=\frac1{Mc}(Q\Sigma\one)^Tx+u$,
and recalling that $A(1)=Z$, yields
$$ (B^\Sigma(1),F_\Sigma)
\stackrel d=
\left(
\sqrt c\,Q \mathbf d(H_0)+Z\one,\;
\lambda_{\max}\!\left(
\sqrt c\left(H_0-\frac{\operatorname{Tr}H_0}{M}I\right)
\right)+Z
\right).$$
This is precisely, 
$$(B^\Sigma(1),F_\Sigma)
\stackrel d=
\left(
\sqrt c\,Q\mathbf d(H_0)+Z\one,\;
\lambda_{\max}\!\left(
\sqrt c\left(H_0-\frac{\operatorname{Tr}H_0}{M}I\right)+ZI
\right)
\right),$$
giving the desired joint identity.

If $c=0$, Proposition~\ref{prop:blockform} gives $\Sigma=a\one\one^T$ for some $a\ge 0$. Let $W$ be a standard scalar Brownian motion independent of $H_0$, and set $Z=\sqrt a\,W(1)$. We may realize $B^\Sigma(t)=\sqrt a\,W(t)\one$, so $F_\Sigma=Z$ and the asserted joint identity follows.
\end{proof}

\begin{cor}\label{cor:conditional}
Let $H=(h_{ij}) \in {\cal H}_M$, where $\mathbf d(H)=0$ and for $1 \leq i < j \leq M$, $h_{ij}$ are independent with $h_{ij} \sim N(0, v_{ij})_{\CC}$. Define $\beta^\Sigma (t) = QB^\Sigma(t) - tQB^\Sigma(1)$. If $Q\Sigma Q = cQ$ for some $c>0$ and $v_{ij}=c$ for every $1 \leq i <j \leq M$, then for every $d \in \RR^M$,
\begin{equation*}
\max_{0 = t_0 \leq t_1 \leq \cdots \leq t_M=1} \sum_{i=1}^M \left [d_i(t_i-t_{i-1}) + \beta_i^\Sigma(t_i) - \beta_i^\Sigma(t_{i-1}) \right ] \stackrel{d} = \lambda_{\textrm{max}} \left ( \diag(d)+H \right ).
\end{equation*}
\end{cor}
\begin{proof}
The process $\beta^\Sigma$ is a Brownian bridge with  $\cov(\beta^\Sigma(s), \beta^\Sigma(t)) = c(s \wedge t -st)Q$, and the traceless matrix in Theorem \ref{thm:qsigmaq} is
 \begin{equation*}
     \sqrt{c} \left ( H_0 - \frac{\Tr H_0}{M}I \right ) = \diag \left ( \sqrt{c} Q \mathbf d(H_0) \right ) + \sqrt{c} \left ( H_0 - \diag (\mathbf d (H_0)) \right ),
 \end{equation*}
where the off-diagonal term is independent of $\sqrt{c} Q \mathbf d (H_0)$ and has the same law as $H$. Conditioning on the endpoint $QB^\Sigma(1)=d\in\one^\perp$, the Brownian motion with covariance matrix $cQ$ has the law of $td + \beta^\Sigma (t)$, while conditioning on $\sqrt{c}Q \mathbf d(H_0) = d$, the matrix $\sqrt{c} \left ( H_0 - \frac{\Tr H_0}{M} I \right )$ has the same law as $\diag (d) + H$. Therefore, the conditional identity holds for $N(0,cQ)$-almost every $d\in\one^\perp$.

For $d, d' \in \one^\perp$ and for every $0 =t_0 \leq \cdots \leq t_M = 1$,
\begin{align*}
\abs{\sum_{i=1}^M(d_i-d_i')(t_i-t_{i-1})} \leq \norm{d-d'}_\infty \sum_{i=1}^M (t_i-t_{i-1}) = \norm{d-d'}_\infty.
\end{align*}
Taking the maximum over admissible choices yields the same bound for the maximal functional, which is thus continuous in $d$. By Weyl's inequality,
\begin{equation*}
\abs{\lambda_{\textrm{max}} (\diag(d) +H) - \lambda_{\textrm{max}} (\diag (d') + H)} \leq \norm{d-d'}_\infty,
\end{equation*}
and the two previous bounds give weak continuity of both laws in $d$. Since the support of $N(0,cQ)$ is $\one^\perp$, continuity extends the $N(0,cQ)$ almost everywhere identity to every $d \in \one^\perp.$ Since $d = Qd + \frac{\one^Td }{M} \one$ adds $\one^Td/M$ to the maximal functional and eigenvalue, it follows that the identity holds for every $d \in \RR^M$.  
\end{proof}  

The rigidity argument rests on comparing the effect of sending one coordinate to $-r$, and as such we first quantify the effect of a large negative penalty on one coordinate of the Brownian last-passage functional.

\begin{lemma}\label{lem:boundary}
Let $X$ and $Y$ be scalar Brownian bridges from $0$ to $0$ on $[0,1]$, with $X$ nondegenerate.  Suppose that, for some $\sigma > 0$,
\begin{equation*}
\cov(Y(s), Y(t)) = \sigma^2 (s \wedge t- st)
\end{equation*}
for every $0 \leq s, t \leq 1$. Then, as $r \to \infty$,
\begin{equation*}
\EE \max_{0 \leq s \leq t \leq 1 } \{X(s) + Y(t) - r(1-t) \} = \EE \max_{ 0 \leq s \leq 1} X(s) + \frac{\sigma^2}{2r} + o(r^{-1}).
\end{equation*}
\end{lemma}
\begin{proof}
Let $F_r = \max_{0 \leq s \leq t \leq 1 } \{X(s) + Y(t) - r(1-t) \}$ and $S_r = \max_{0 \leq t \leq 1 } \{ Y(t) - r(1-t) \}$. Since $t=1$ is admissible and $Y(1)=0$, it follows that
\begin{equation*}
0 \leq F_r - \max_{0 \leq s \leq 1} X(s) \leq S_r.
\end{equation*}
The process $Y(1-t)-rt$ is a Brownian bridge from $0$  to $-r$ with variance parameter $\sigma^2$, so for $x \geq 0$,
\begin{equation*}
\PP(S_r > x) = \exp \left ( - \frac{2x(x+r)}{\sigma^2} \right )
\end{equation*}
by the Brownian bridge reflection principle. Integrating the tail and applying the inequality $1 - \exp(-u) \leq u$ gives $0 \leq \sigma^2/(2r) - \EE S_r \leq \sigma^4/(2r^3)$ and $\EE(rS_r)^2 \leq \sigma^4/2.$

Denote the first time $X$ attains its maximum and the first time the maximum defining $S_r$ is attained by $\tau$ and $\tau_r$ respectively. On $\{\tau \leq \tau_r\}$, necessarily $F_r = \max_{0 \leq s \leq 1}X(s) + S_r$ as the universal upper bound $F_r\le \max X+S_r$ holds,  
while $(\tau,\tau_r)$ is admissible on that event and gives the reverse inequality.  Hence, 
\begin{equation*}
0 \leq S_r - \left ( F_r - \max_{0 \leq s \leq 1}X(s) \right ) \leq S_r \one_{\{\tau_r < \tau\}}.
\end{equation*}

Next, $Y(\tau_r) - r(1-\tau_r)=S_r \geq 0$, hence $0 \leq 1 - \tau_r \leq r^{-1} \max_{0 \leq t \leq 1}Y(t)$ and thus $\tau_r \to 1$ almost surely. Moreover, since $X(1)=0$, the strict positivity of the maximum gives $\tau < 1$ almost surely.  
 
Thus, $\PP(\tau_r < \tau) \to 0$, and by the Cauchy--Schwarz inequality,
\begin{align*}
0 \le r \EE S_r - r \EE \left ( F_r - \max_{0 \leq s \leq 1}X(s) \right ) 
 \leq \left ( \EE(rS_r)^2 \PP(\tau_r < \tau) \right )^{1/2} \longrightarrow 0.
\end{align*}
\end{proof}

We next obtain the analogous first-order expansion for the maximal eigenvalue of a $3\times 3$ Gaussian Hermitian matrix when one diagonal entry tends to $-\infty$.

\begin{lemma}\label{lem:matrixboundary}
Let $h_{12}, h_{13}$ and $h_{23}$ be independent, centered complex Gaussian random variables with $h_{12}$ nondegenerate. Then, as $r \to \infty$,
\begin{equation*}
\EE \lambda_{\textrm{max}} 
\begin{pmatrix}
0 & h_{12} & h_{13} \\
\overline{h_{12}} & 0 & h_{23} \\
\overline{h_{13}} & \overline{h_{23}} & -r
\end{pmatrix}
= \EE \abs{h_{12}} + \frac{\EE \abs{h_{13}}^2 + \EE\abs{h_{23}}^2}{2r} + o(r^{-1}).
\end{equation*}
\end{lemma}
\begin{proof}
Let $\lambda_r$ be the largest eigenvalue of the above Hermitian matrix. 
Expanding the characteristic determinant along the third column gives 
\begin{align*}
(\lambda_r+r)(\lambda_r- \abs{h_{12}})(\lambda_r+\abs{h_{12}}) &= \lambda_r (\abs{h_{13}}^2 + \abs{h_{23}}^2) + 2 \Re (h_{12} \overline{h_{13}}h_{23})\\
& \leq \lambda_r(\abs{h_{13}}^2 + \abs{h_{23}}^2 ) + 2 \abs{h_{12}\overline{h_{13}}h_{23}} \\
& \leq \lambda_r ( \abs{h_{13}}^2 + \abs{h_{23}}^2) + \abs{h_{12}}( \abs{h_{13}}^2 + \abs{h_{23}}^2)\\
& \leq ( \lambda_r + \abs{h_{12}}) (\abs{h_{13}}^2+ \abs{h_{23}}^2).
\end{align*}
Dividing both sides by $(\lambda_r+r)( \lambda_r + \abs{h_{12}})/r$ and noting that by the Cauchy interlacing theorem, $\lambda_r \geq \abs{h_{12}}$ gives
\begin{equation*}
0 \leq r( \lambda_r - \abs{h_{12}}) \leq \frac{r(\abs{h_{13}}^2 + \abs{h_{23}}^2)}{\lambda_r + r} \leq \abs{h_{13}}^2 + \abs{h_{23}}^2,
\end{equation*}
so $\lambda_r \to \abs{h_{12}}$ almost surely and so, ${r}/(\lambda_r+r) \to 1$, as $r \to \infty$. Noting $\abs{h_{12}}> 0$,
\begin{align*}
r ( \lambda_r - \abs{h_{12}}) &= \frac{r\lambda_r (\abs{h_{13}}^2 + \abs{h_{23}}^2) + 2 r\Re (h_{12} \overline{h_{13}}h_{23})}{(\lambda_r+r)(\lambda_r+ \abs{h_{12}})}\\
& \longrightarrow \frac{\abs{h_{13}}^2 + \abs{h_{23}}^2}{2} + \frac{\Re(h_{12}\overline{h_{13}}h_{23})}{\abs{h_{12}}}, 
\end{align*}
as $r \to \infty$.   
Moreover, the preceding bound gives, for every $r>0$,
$
0\le r\bigl(\lambda_r-|h_{12}|\bigr)
\le |h_{13}|^2+|h_{23}|^2$. 
Since $\mathbb E\bigl(|h_{13}|^2+|h_{23}|^2\bigr)<\infty$,
the family $r(\lambda_r-|h_{12}|)$ is dominated by an integrable random variable independent of $r$. Therefore, by the dominated convergence theorem, $\lim_{r\to\infty}
r\left(\mathbb E\lambda_r-\mathbb E\abs{h_{12}}\right)= \mathbb E\!\left[
\lim_{r\to\infty}r(\lambda_r-|h_{12}|)
\right]$. 
Using the almost-sure limit obtained above, 
$$\lim_{r\to\infty}r(\lambda_r-|h_{12}|) = 
\frac{|h_{13}|^2+|h_{23}|^2}{2}
+ \frac{\Re(h_{12}\overline{h_{13}}h_{23})}{|h_{12}|},$$
we obtain 
$\lim_{r\to\infty}
r\left(\mathbb E\lambda_r-\mathbb E|h_{12}|\right)
= \frac{{\mathbb E}|h_{13}|^2+{\mathbb E}|h_{23}|^2}{2},$ 
because, by independence and centering, $
{\mathbb E}\left[
\left.
\frac{\Re(h_{12}\overline{h_{13}}h_{23})}{|h_{12}|}
\right|h_{12}
\right]=0$ a.s. 
Thus $\mathbb E\lambda_r=\mathbb E|h_{12}|
+\frac{\mathbb E|h_{13}|^2+\mathbb E|h_{23}|^2}{2r}
+o(r^{-1})$.  
\end{proof}

We now use these asymptotic comparisons to show that a largest eigenvalue representation valid for every endpoint forces projected isotropy and a common off-diagonal variance.  This is presented through a two-stage rigidity argument as follows:

(i) two-coordinate degeneration gives $v_{ij}=\frac12(e_i-e_j)^T\Sigma(e_i-e_j)$; 

(ii) three-coordinate degeneration combined with the $r^{-1}$ asymptotics of the previous two lemmas forces $v_{ij}=v_{ik}=v_{jk}$.

\noindent
The resulting quadratic-form identity then yields
$x^T\Sigma x=c\|x\|^2$, $x\in\one^\perp$, 
and thus $Q\Sigma Q=cQ$.

\begin{thm}\label{thm:mainclassification}
Let $M \geq 2$ and let $Q \Sigma Q$ be positive definite on $\one^\perp$. Let $H=(h_{ij})_{1\le i,j \le M} \in {\cal H}_M$, where $\mathbf d(H)=0$ and for $1 \leq i < j \leq M$, $h_{ij}$ are independent with $h_{ij} \sim N(0, v_{ij})_{\CC}$. Define $\beta^\Sigma (t) = QB^\Sigma(t) - tQB^\Sigma(1).$ Then, 
\begin{equation*}
\max_{0 = t_0 \leq t_1 \leq \cdots \leq t_M=1} \sum_{i=1}^M \left [d_i(t_i-t_{i-1}) + \beta_i^\Sigma(t_i) - \beta_i^\Sigma(t_{i-1}) \right ] \stackrel{d} = \lambda_{\textrm{max}} \left ( \diag(d)+H \right )
\end{equation*}
for every $d \in \RR^M$ if and only if there exists $c> 0$ such that $Q \Sigma Q = cQ$ and $v_{ij} = c$ for every $1 \leq i < j \leq M$.
\end{thm}
\begin{proof}
First, suppose that the identity holds for every $d \in \RR^M.$ Fix $i<j$, let $d_i = d_j=0$, let $d_\ell = -r$ for $\ell \not \in \{i, j\}$, and choose $0 = t_0^{(r)} \leq \cdots \leq t_M^{(r)} = 1$ so the maximum is attained. For $0 \leq t \leq 1$, setting $t_0 = \cdots = t_{i-1}=0$, $t_i=\cdots = t_{j-1} = t$, and $t_j = \cdots = t_M=1$ is admissible and yields $\beta_i^\Sigma(t) - \beta_j^\Sigma(t)$, so the maximal functional is nonnegative. At the chosen times, bounding the bridge increments yields
\begin{equation*}
0 \leq r \sum_{ \ell \not \in \{i, j\}} (t_\ell^{(r)} - t_{ \ell - 1}^{(r)} ) \leq 2 \sum_{\ell =1}^M \max_{ 0 \leq u \leq 1} \abs{\beta_\ell^\Sigma(u)},
\end{equation*}
and thus $\sum_{\ell \not \in \{i, j \}}(t_\ell^{(r)}-t_{\ell-1}^{(r)}) \to 0$. The set of $(t_0, \ldots, t_M)$ satisfying $0=t_0 \leq \cdots \leq t_M=1$ is compact, hence for any sequence $r \to \infty$, there is a convergent subsequence  of the chosen times with limit of the form $t_0 = \cdots = t_{i-1}=0$, $t_i = \cdots = t_{j-1} =t$, and $t_j= \cdots= t_M=1$. By continuity, the bridge increments along this subsequence converge to $\beta_i^\Sigma(t) - \beta_j^\Sigma(t)$, and since $-r \sum_{\ell \not \in \{i, j\}}(t_\ell^{(r)} -t_{\ell-1}^{(r)}) \leq 0$, 
\begin{eqnarray*}
\limsup_{r \to \infty} \max_{0=t_0 \leq t_1 \leq \cdots \leq t_M=1}
\sum_{k=1}^M\left (d_k(t_k-t_{k-1})+\beta_k^\Sigma(t_k)-\beta_k^\Sigma(t_{k-1})\right ) \\
\le \max_{0 \le t \leq 1}(\beta_i^\Sigma(t) - \beta_j^\Sigma(t)).
\end{eqnarray*}
The reverse inequality for every $r$ follows by varying $t$. Thus, the maximal functional converges almost surely to this maximum.

Let $\lambda_r = \lambda_{\textrm{max}}(\diag(d) + H)$. Noting that $\diag(d)$ is negative semidefinite, applying the Cauchy interlacing theorem and eigenvalue monotonicity gives $\abs{h_{ij}} \leq \lambda_r \leq \lambda_{\textrm{max}}(H)$. If $M= 2$, then $\lambda_r = \abs{h_{12}}$.  If $M \geq 3$, symmetrically permute rows and columns and let $A  = \begin{pmatrix}
0 & h _{ij} \\
\overline{h_{ij}}  & 0
\end{pmatrix}$
and $\diag(d) +H = \begin{pmatrix}
A & C\\
C^* & D-rI_{M-2}
\end{pmatrix}$, where $C$ and $D$ are independent of $r$. For $r> \norm{D}$, the lower right block is invertible and then the Schur complement equation for $\lambda_rI - (\diag(d)+H)$ is
\begin{equation*}
\det \left ( \lambda_r I_2 - A -C [(r + \lambda_r)I_{M-2}-D]^{-1}C^* \right ) =0,
\end{equation*}
and because $\lambda_r \geq 0$,
\begin{equation*}
\norm{C [(r + \lambda_r)I_{M-2}-D]^{-1}C^*} \leq \frac{\norm{C}^2}{r- \norm{D}} \longrightarrow 0.
\end{equation*}
Since $A$ has eigenvalues $\pm \abs{h_{ij}}$  and $\lambda_r \geq \abs{h_{ij}}$,  the determinant forces $\lambda_r$ to have a unique limit point $\abs{h_{ij}}$, so $\lambda_r \to \abs{h_{ij}}$ almost surely. Letting $r  \to \infty$, 
\begin{equation*}
\max_{0 \leq t \leq 1} \left ( \beta_i^\Sigma(t) - \beta_j^\Sigma(t) \right ) \stackrel{d} = \abs{h_{ij}}.
\end{equation*}
The bridge $\beta_i^\Sigma-\beta_j^\Sigma$ has variance parameter $(e_i-e_j)^T \Sigma(e_i-e_j)>0$ since $e_i-e_j \in \one^\perp$, and the reflection formula shows the tail is  $\exp(-2y^2/((e_i-e_j)^T \Sigma(e_i-e_j)))$ for $y \geq 0$. Therefore,
\begin{align*}
\EE \left [ \left ( \max_{0 \leq t \leq 1} \left ( \beta_i^\Sigma(t) - \beta_j^\Sigma(t) \right ) \right)^2 \right ] & = 2 \int_0^\infty y \exp \left ( - \frac{2y^2}{(e_i-e_j)^T \Sigma(e_i-e_j)} \right ) dy \\
&= \frac{1}{2}(e_i-e_j)^T \Sigma(e_i-e_j).
\end{align*}
Since $\EE \abs{h_{ij}}^2 = v_{ij}$, equality in distribution yields
\begin{equation*}
v_{ij} = \frac{1}{2}(e_i-e_j)^T \Sigma(e_i-e_j) > 0
\end{equation*}
for every $1 \leq i < j \leq M.$ In the case $M=2$, let $c = v_{12}>0$.

Next, suppose $M \geq 3$,  fix $i < j < k$, and fix $d_i, d_j, d_k$. Let $d_ \ell  =-r$ for $\ell \not \in \{i, j, k\}$ and choose $0 = t_0^{(r)} \leq \cdots \leq t_M^{(r)} = 1$ so that the maximum, $F_r$, is attained.  Let $\delta_r:=\sum_{\ell\notin\{i,j,k\}}
(t_\ell^{(r)}-t_{\ell-1}^{(r)})$.  Then, pathwise, 
$r\delta_r\le C(\beta^\Sigma,d_i,d_j,d_k)$,
where the right-hand side is finite and independent of $r$. 
Indeed, comparison with any admissible choice having all increments outside 
$\{i,j,k\}$ equal to zero, together with the uniform bound on the bridge increments, gives this estimate. 
Thus $\delta_r\to 0$ almost surely. 

By compactness of the admissible simplex, every sequence of maximizing
vectors has a convergent subsequence. Since every increment
$t_\ell^{(r)}-t_{\ell-1}^{(r)}, \ell\notin\{i,j,k\}$, is bounded by
$\delta_r$, all such increments vanish along any limiting subsequence.
Thus every limit is determined by two parameters $0\le s\le t\le 1 $.
Using continuity of the bridge paths and the fact that the penalization term
$-r\delta_r \le 0$, we obtain
$$
\limsup_{r\to\infty}F_r
\leq
\max_{0\leq s\leq t\leq1}
\left\{
d_i s+d_j(t-s)+d_k(1-t)
+(\beta_i^\Sigma-\beta_j^\Sigma)(s)
+(\beta_j^\Sigma-\beta_k^\Sigma)(t)
\right\}.
$$
Conversely, every $0\le s\le t\le 1$ can be realized by an admissible choice with all increments outside $\{i,j,k\}$  equal to zero, so the reverse $\liminf$  inequality holds. Therefore, the maximal functional converges almost surely to
\begin{equation*}
\max_{0 \leq s \leq t \leq 1} \{ d_i s + d_j(t-s) + d_k(1-t) + ( \beta_i^\Sigma -\beta_j^\Sigma)(s) + (\beta_j^\Sigma - \beta_k^\Sigma) (t) \}.
\end{equation*}

Apply the same Schur-complement argument to the principal $3\times 3$ block indexed by $i,j,k$, after a simultaneous permutation of rows and columns.  
The Cauchy interlacing theorem and $\lambda_{\textrm{max}}(\diag(d)+H) \leq \lambda_{\textrm{max}}(H)+\max \left \{ d_i, d_j, d_k, 0 \right \}$ bound the maximum eigenvalue uniformly in $r$, so the correction term in the Schur complement is $O(r^{-1})$, while the determinant and interlacing prove convergence to the minor's largest eigenvalue. Then, for every $d_i, d_j, d_k \in \RR$,
\begin{equation*}
\max_{0 \leq s \leq t \leq 1} \{ d_i s + d_j(t-s) + d_k(1-t) + (\beta_i^\Sigma - \beta_j^\Sigma)(s) + (\beta_j^\Sigma - \beta_k^\Sigma)(t) \}
\end{equation*}
has the same law as $ \lambda_{\textrm{max}}
\begin{pmatrix}
d_i & h_{ij} & h_{ik}\\
\overline{h_{ij}} & d_j & h_{jk} \\
\overline{h_{ik}} & \overline{h_{jk}} & d_k
\end{pmatrix}$.

The variance parameter of $\beta_j^\Sigma -\beta_k^\Sigma$ for $(d_i, d_j, d_k) = (0, 0 , -r)$  is $2v_{jk}$. By Lemma \ref{lem:boundary} with $X = \beta_i^\Sigma- \beta_j^\Sigma$ and $Y = \beta_j^\Sigma - \beta_k^\Sigma$ and Lemma \ref{lem:matrixboundary} with $(h_{12}, h_{13}, h_{23}) = (h_{ij}, h_{ik}, h_{jk})$,
\begin{equation*}
\EE \max_{0 \leq t \leq 1} [\beta_i^\Sigma(t) - \beta_j^\Sigma(t) ] + \frac{v_{jk}}{r} + o(r^{-1}) = \EE \abs{h_{ij}} + \frac{v_{ik} + v_{jk}}{2r} + o(r^{-1}).
\end{equation*}
Next, consider $(d_i, d_j, d_k) = (-r, 0,0).$ Since $\cov(\beta^\Sigma(1-s), \beta^\Sigma (1-t)) = (s \wedge t -st) Q \Sigma Q$, the process $\beta^\Sigma$ is invariant under time reversal, and after changing variables $(s,t) \mapsto (1-t,1-s)$, it follows that
\begin{equation*}
\max_{0 \leq s \leq t \leq 1} \{ - rs + ( \beta_i^\Sigma- \beta_j^\Sigma)(s) + (\beta_j^\Sigma - \beta_k^\Sigma) (t) \}
\end{equation*}
has the same law as
\begin{equation*}
\max_{0 \leq s \leq t \leq 1} \{(\beta_j^\Sigma - \beta_k^\Sigma)(s) + (\beta_i^\Sigma- \beta_j^\Sigma)(t) - r(1-t) \}.
\end{equation*}
After reordering the rows and columns from $i,j,k$ to $k,j,i$, Lemma \ref{lem:boundary}
 and Lemma~\ref{lem:matrixboundary}
 give $v_{ij}/r$ and $(v_{ik}+v_{ij})/(2r)$  as the $r^{-1}$ terms in the expectations, so
 \begin{equation*}
\EE \max_{ 0 \leq t \leq 1} [\beta_j^\Sigma(t) - \beta_k^\Sigma(t) ]+ \frac{v_{ij}}{r} + o(r^{-1}) = \EE \abs{h_{jk}} + \frac{v_{ij} + v_{ik}}{2r} + o(r^{-1}).
\end{equation*}
The identities for $(i,j)$ and $(j,k)$ show the constant terms agree. Subtracting the common constant term in each expansion, multiplying by $r$ and taking the limit as $r \to \infty$ gives $v_{jk} = (v_{ik}+v_{jk})/2$ and $v_{ij} = (v_{ij} + v_{ik})/2$, hence $v_{ij}=v_{ik}= v_{jk}.$ Let $c =v_{ij}>0$. Thus, for $M\ge 2$, $v_{ij} =(e_i-e_j)^T\Sigma(e_i-e_j)/2 = c > 0$, for all $i<j$. Next, for $x\in\one^\perp$, 
\begin{align*}
x^T\Sigma x= -\frac12\sum_{i,j=1}^M
x_i x_j\,(e_i-e_j)^T\Sigma(e_i-e_j)=-2c\sum_{1\le i<j\le M}x_i x_j =c\|x\|^2.
\end{align*}
Hence, by polarization, $x^T\Sigma y=c\,x^Ty$ for all
$x,y\in\one^\perp$. Equivalently, the restriction of $Q\Sigma Q$
to $\one^\perp$ is $cI_{\one^\perp}$. Since both $Q\Sigma Q$ 
and $cQ$ annihilate $\one$, it follows that $Q\Sigma Q=cQ$.

Conversely, if $Q \Sigma Q = cQ$ and $v_{ij} = c>0$ for all $i<j,$  the result follows from Corollary \ref{cor:conditional}.
\end{proof}

To finish this section, we characterize the transition matrices of Markov chains with uniform stationary distribution whose asymptotic covariance satisfies
$\Sigma=c\left(I-\frac1M{\bf 1}{\bf 1}^T\right)$. 
This extends the characterization obtained in \cite[Section 4]{HL2} and provides a constructive description in terms of a skew-symmetric operator on ${\bf 1}^\perp$.  For the asymptotic covariance matrix of a Markov chain with uniform stationary distribution, one has $\Sigma\one =0$, and hence $Q\Sigma Q=\Sigma$. Therefore, the isotropy condition $Q\Sigma Q=cQ$ of the previous results reduces in this setting to $\Sigma=cQ$.  
We then conclude by translating the condition $\Sigma=cQ$ into an explicit parametrization of the irreducible, aperiodic Markov transition matrices with uniform stationary distribution that realize it.

\begin{thm}\label{thm:expandedclassification}
Fix $M \geq 2$ and $c>0$, and let $U: \RR^{M-1} \to \one^\perp$ be an isometry. The $M \times M$ matrix $\mathsf P$ with nonnegative entries is the transition matrix of an irreducible aperiodic Markov chain with uniform stationary distribution and asymptotic covariance matrix 
\begin{equation}\label{SigmaStationary}
\Sigma =c \left ( I - \frac{1}{M} \one \one^T \right ),
\end{equation}
if and only if 
\begin{equation}\label{characterizationP}
\mathsf P = \frac{1}{M} \one \one^T+ U \left [ I - \left (\frac{1+Mc}{2}I +\mathsf K \right )^{-1} \right ]U^T,
\end{equation}
for some real skew-symmetric matrix $\mathsf K \in \RR^{(M-1)\times (M-1)}.$
\end{thm}
\begin{proof}
Let $\Sigma$ have the form given in \eqref{SigmaStationary} and let $Z_i$ be the indicator vector of the state at time $i$ for the chain started from its stationary distribution. Then, $\cov (Z_0,Z_j) = \frac{1}{M} (\mathsf P^j - \frac{1}{M} \one \one^T)$ and $\cov(Z_j, Z_0) = \frac{1}{M} ( (\mathsf P^T)^j - \frac{1}{M} \one \one^T)$ for positive integer $j$, so by definition,
\begin{equation*}
\Sigma = \frac{1}{M} \left ( I - \frac{1}{M} \one \one^T \right ) + \frac{1}{M} \sum_{j \geq 1} \left [ \mathsf P^j +(\mathsf P^T)^j - \frac{2}{M} \one \one^T \right ].
\end{equation*}
Restricting to $\one^\perp$ and summing the geometric series,
\begin{align*}
MU^T \Sigma U &= I +\sum_{j \geq 1} \left [ (U^T \mathsf P U)^j + ((U^T \mathsf P U)^T)^j \right]\\
&= \left ( I - U^T \mathsf P U \right )^{-1} +\left (I - (U^T \mathsf P U)^T \right )^{-1}-I.
\end{align*}
Since $U^T \Sigma U = cI$,
\begin{equation*}
\left ( I - U^T \mathsf P U \right )^{-1} + \left ( I - (U^T \mathsf P U)^T \right)^{-1} = (1+Mc)I,
\end{equation*}
so 
\begin{equation*}
\left ( I - U^T \mathsf P U  \right)^{-1} = \frac{1+Mc}{2} I + \mathsf K, 
\end{equation*}
for some real skew-symmetric matrix $\mathsf K$. Finally, since the stationary distribution is the uniform one, $\mathsf P$ is then doubly stochastic and so $\mathsf P = \frac{1}{M} \one \one^T + U(U^T \mathsf P U)U^T$. Solving for $\mathsf P$ gives \eqref{characterizationP}.

Conversely, assume \eqref{characterizationP} and let $R = (1+Mc)I/2 +\mathsf K$. Since $c>0$ and $\mathsf K$ is skew-symmetric, $R$ is invertible and
\begin{align*}
I - (I-R^{-1})^T (I-R^{-1}) & = R^{-T}(R + R^T -I) R^{-1}\\
& = McR^{-T} R^{-1}
\end{align*}
is positive definite, so $\norm{I-R^{-1}}_2 < 1$. Moreover, $\mathsf P \one = \mathsf P^T \one = \one$ and
\begin{equation*}
\mathsf P^n = \frac{1}{M} \one \one^T + U(I-R^{-1})^n U^T \longrightarrow \frac{1}{M} \one \one^T. 
\end{equation*}
Since $\mathsf P$ is assumed entrywise nonnegative and $\mathsf P\mathbf1=\mathbf 1$, it is a transition matrix. Since also $\mathsf P^T\one = {\bf 1}$, it is doubly stochastic.  Moreover, $\mathsf P^n$ is entrywise positive for all sufficiently large $n$, and its stationary distribution is uniform. Therefore, the chain is irreducible and aperiodic with uniform stationary distribution. Applying the earlier covariance identity with $U^T \mathsf P U = I - R^{-1}$ gives $MU^T \Sigma U = R+R^T-I=McI$. Since $\Sigma$ is symmetric and $\Sigma \one =0$, using $UU^T = I - \frac{1}{M} \one \one ^T$ gives \eqref{SigmaStationary}. 
    
\end{proof}

\section{Counterexamples to Random-Matrix Representability}

The preceding section identifies a natural covariance structure under which the limiting Brownian functional associated with a Markov random word admits a Gaussian random-matrix representation. We now show that, outside this regime, the failure can be substantially stronger than the absence of a representation within the particular Gaussian ensemble considered above: for certain Markov chains, the limiting functional cannot have the law of the maximal eigenvalue of any centered jointly Gaussian Hermitian matrix whose dimension equals the alphabet size.

The random-matrix paradigm can therefore genuinely fail outside the isotropic situation, showing that the Markovian limiting laws exhibit phenomena beyond those captured by GUE-like representations.
The obstruction is based on small-deviation probabilities. The event that the Brownian last-passage functional is small forces an associated Brownian motion to remain inside a suitable cone while ending in a small neighborhood of the origin. Estimates for Brownian motion in cones then give an upper bound on this probability that decays too rapidly to be compatible with the small-ball behavior of a Gaussian random matrix.  We first carry out this argument for an explicit four-state cyclic chain and then show that the same phenomenon already occurs for a non-cyclic three-state chain.

The counterexamples rely on a small-deviation estimate for Brownian motion confined to a wedge and ending near its vertex, which we record next.

\begin{lemma}\label{lem:weyl-small-ball}
Let $\mathcal D \subset \RR^2$ be an open wedge with vertex at the origin and opening $\theta$ with $0 < \theta < 2\pi$. For standard Brownian motion $B$ in $\RR^2$, let $\PP_x$ be its law when started at $x$. Let $x_0 \in \mathcal D$, and let $E \subset \RR^2$ be a bounded Borel set. Then 
\begin{equation}\label{firstboundary}
\PP_{\varepsilon x_0} \left ( B(s) \in \overline{\mathcal D} (0 \le s \le 1), B(1) \in \varepsilon E \right ) = O ( \varepsilon^{2(1 +\pi/\theta)} )
\end{equation}
as $\varepsilon \downarrow 0$. Since $\cal D$ is a cone, the endpoint condition on the event above may equivalently be written
$B(1)\in\varepsilon(E\cap {\cal D})$.
\end{lemma}
\begin{proof}
Choose $0< R < +\infty$ such that $E\subset B(0,R)$.
Write $\varepsilon x_0=\rho\omega_0$, and $y=r\omega$. 
For $y\in\varepsilon(E\cap {\cal D})$, 
$\rho=\varepsilon|x_0|$ and $r\le R\varepsilon$. 
In dimension two, the parameters in the cone heat-kernel expansion of \cite[equation (2.2)]{BS} are $\alpha_j={j\pi}/{\theta}$, $j\ge 1$.
Using the Bessel-function estimate in \cite[equation (2.4)]{BS} therefore gives, uniformly for $y\in\varepsilon(E\cap {\cal D})$, 
$$p_{\cal D}(1,\varepsilon x_0,y) = O\!\left((\rho r)^{\pi/\theta}\right) = O\!\left(\varepsilon^{2\pi/\theta}\right),$$
i.e., the killed transition density in $\mathcal D$ at time one, from $\varepsilon x_0$ to $y\in \varepsilon(E\cap {\cal D})$, is uniformly $O(\varepsilon^{2\pi/\theta})$. Since $\varepsilon(E\cap {\cal D})$ has area $O(\varepsilon^2)$, integration over the endpoint yields \eqref{firstboundary}. 
For later use, we also note that replacing $\cal D$ by its closure in the path condition does not change the probability.  Indeed, by the strong Markov property, if a wall is first hit before time $1$, the corresponding normal coordinate is a one-dimensional Brownian motion started at zero and almost surely does not remain on one side of zero for any positive time interval.  Now, a first hit at time $1$ is also a null event since $\partial \mathcal D$ has Lebesgue measure zero and  $B(1)$ has  a density. Thus, the open and closed wedge path events agree (up to a null event).
\end{proof}

We first apply the lemma to an explicit four-state cyclic chain.  
Let $\Pi_4 = \frac{1}{4} \one_4 \one_4^T$ and $\zeta = \frac{1}{2} (1, -1, 1, -1)^T$. Consider the Markov chain $(Y_n)_{n \ge 0}$ on $1< 2< 3<4$, started from its uniform stationary distribution, with transition matrix

\begin{equation}\label{eq:4transition}
\mathsf P = \Pi_4 - \frac{15}{17} \zeta \zeta^T = \frac{1}{34} 
\begin{pmatrix}
1&16&1&16\\
16&1&16&1\\
1&16&1&16\\
16&1&16&1\\
\end{pmatrix}.  
\end{equation}
Observe that $\mathsf P$ is symmetric, cyclic, doubly stochastic, with positive entries and eigenvalues $-\frac{15}{17}, 0, 0$ on $\one_4^\perp.$ Since $\mathsf P^k- \Pi_4=(-15/17)^k \zeta \zeta^T$ for $k \geq 1$, the standard asymptotic covariance formula, together with stationarity and symmetry, yields 
\begin{align}
\Sigma_*&= \frac{1}{4} (I_4 - \Pi_4) + \frac{1}{2} \sum_{k \geq 1}( \mathsf P^k-\Pi_4) \notag \\
&= \frac{1}{4} (I_4 - \Pi_4) - \frac{15}{64} \zeta \zeta^T \notag \\
&= \frac{1}{256}
\begin{pmatrix}
33&-1&-31&-1\\
-1 & 33 & -1 & -31 \\
-31 & -1 & 33 & -1 \\
-1 & -31 & -1 & 33
\end{pmatrix}.\label{eq:sigmastar}
\end{align}
Let $B^{\Sigma_*}$ be Brownian motion in $\RR^4$ with covariance matrix $\Sigma_*$, let $X_j = B_j^{\Sigma_*} - B_{j+1}^{\Sigma_*}$ for $1 \leq j \leq 3$, and let $F=B_4^{\Sigma_*}(1)+\max_{0 \leq t_1 \leq t_2 \leq t_3 \leq 1} \sum_{j=1}^3 X_j(t_j)$. Denote the length of the longest weakly increasing subsequence of $Y_1, \ldots, Y_N$ by $\text{LI}_N$. By \cite[Corollary 3.1]{HL2},
\begin{equation*}
\frac{\text{LI}_N-N/4}{\sqrt{N}} \Longrightarrow F.
\end{equation*}

The last-passage functional given by $(Y_n)_{n\ge 0}$ cannot have the law of the largest eigenvalue of any centered Gaussian Hermitian matrix of the same size.

\begin{thm}\label{thm:4counterexample}
For the chain with transition matrix given in \eqref{eq:4transition}, there is no centered jointly Gaussian Hermitian random matrix $H \in {\cal H}_4$ such that $\lambda_{\textrm{max}}(H)$ has the same law as $F$.
\end{thm}
\begin{proof}
Since $\Sigma_* \one_4 = 0$, it follows that $\sum_{i=1}^4 B_i^{\Sigma_*} \equiv 0$ almost surely. For $1 \le i \le 4$, taking $t_k= 0 $ for $k< i$ and $t_k=1$ for $k \geq i$ shows that $F \geq B_i^{\Sigma_*}(1).$ Thus, $B_i^{\Sigma_*}(1) \le \varepsilon$ on $\{F \le \varepsilon\}$ for every $i$, and $B_i^{\Sigma_*}(1) \ge -3 \varepsilon$ because $\sum_{i=1}^4 B_i^{\Sigma_*}(1) =0$. 

Similarly, for $1 \le j \le 3$ and $0 \le s \le 1$, taking $t_k = 0$ for $k<j$, $t_j=s$, and $t_k=1$ for $k>j$ shows that $F \geq B_{j+1}^{\Sigma_*}(1)+X_j(s)$, and consequently $X_j(s) \leq 4 \varepsilon$ on $\{F \le \varepsilon \}$ for 
$1 \le j \le 3$ and $0 \le s \le 1$.  Therefore, $-4 \varepsilon \le X_j(1) \le 4 \varepsilon$ for $ 1 \le j \leq 3$. 
For $0 \le t \leq 1$, let
\begin{equation*}
W_{\varepsilon}(t) = 
\begin{pmatrix}
32 \varepsilon - 4 X_1(t) -4X_3(t) \\
-X_1(t) + X_3(t)\\
16 \varepsilon - X_1(t) - 2X_2 (t) -X_3(t)
\end{pmatrix}.
\end{equation*}
Using \eqref{eq:sigmastar},
\begin{equation*}
(\cov (X_i(s), X_j(t)))_{1 \leq i, j \leq 3} = \frac{s \wedge t}{64} \begin{pmatrix}
17 & -1 & -15 \\
-1 & 17 & -1 \\
-15 & -1 & 17
\end{pmatrix},
\end{equation*}
so $\cov(W_{\varepsilon}(s), W_{\varepsilon}(t)) = (s \wedge t)I_3$ and thus $W_{\varepsilon}$ is a Brownian motion in $\RR^3$ with identity covariance matrix, started at $\varepsilon (32, 0, 16)^T$.  On $\{F \le \varepsilon \}$,
\begin{align*}
W_{\varepsilon, 1}(s) &= 4(4\varepsilon - X_1(s)) + 4(4 \varepsilon -X_3(s))\\
& \geq 4 \abs{(4 \varepsilon - X_1(s)) - (4 \varepsilon -X_3(s))} \\
&= 4 \abs{W_{\varepsilon, 2}(s)},
\end{align*}
and $W_{\varepsilon, 3}(s) = (4 \varepsilon - X_1(s)) + 2(4 \varepsilon -X_2(s)) + (4 \varepsilon -X_3(s)) \ge 0$. 
So on $\{F \le \varepsilon\}$, $W_{\varepsilon}(s)$ stays in the cone $\{z \in \RR^3: z_1 \ge 4 \abs{z_2}, z_3 \ge 0\}$ for $0 \le s \le 1$, and  $W_{\varepsilon} (1) \in \varepsilon([0,64] \times [-8,8] \times [0,32])$.

The wedge in this cone has opening $\theta=\arccos(15/17)<\pi/6$. 
Hence $2\left(1+\frac{\pi}{\theta}\right)>14$, 
and Lemma~\ref{lem:weyl-small-ball} therefore gives, in particular, an $O(\varepsilon^{14})$ bound for the first two coordinates. Next, observe that $16\varepsilon-W_{\varepsilon,3}(t)=X_1(t)+2X_2(t)+X_3(t)$. 
On $\{F\leq\varepsilon\}$, the maximum of the right-hand side lies in
$[0,16\varepsilon]$. Moreover, the endpoint bounds
$-4\varepsilon\leq X_j(1)\leq4\varepsilon$, $1\le j\le 3$, 
imply that
$$
16\varepsilon-W_{\varepsilon,3}(1) = X_1(1)+2X_2(1)+X_3(1)
\in[-16\varepsilon,16\varepsilon].
$$
Set $Z(t):=16\varepsilon-W_{\varepsilon,3}(t)$.  Then $Z$ is a standard one-dimensional Brownian motion started at $0$.
By~\cite[equation~(1.9)]{Shepp}, the joint density of $\left(\max_{0\le t\le1}Z(t),Z(1)\right)$ is $$p(y,x) = \frac{2(2y-x)}{\sqrt{2\pi}}\exp\!\left(-\frac{(2y-x)^2}{2}\right),$$ 
$y>0,\ x<y$.
On $0\le y\le 16\varepsilon, -16\varepsilon\le x\le y$, this density is $O(\varepsilon)$, while the region has area $O(\varepsilon^2)$. Hence the corresponding probability is $O(\varepsilon^3)$ for the third coordinate.  
Since $W_\varepsilon$ has independent standard Brownian coordinate increments, the process $(W_{\varepsilon,1},W_{\varepsilon,2})$ is independent of
the coordinate process $W_{\varepsilon,3}$. Multiplying the two bounds therefore yields
$\mathbb P(F\le\varepsilon)=O(\varepsilon^{17})$.  

Suppose $\lambda_{\textrm{max}}(H)$ has the same law as $F$ for some centered jointly Gaussian Hermitian random matrix $H$ in ${\cal H}_4$. If $H=0$ almost surely, then $\PP(\lambda_{\textrm{max}}(H) \le \varepsilon) = 1$ for every $\varepsilon >0$, which is impossible since $\PP(F \le \varepsilon ) = O(\varepsilon^{17})$.  Otherwise, let $V$ be the smallest linear subspace of $\mathcal H_4$ containing $H$ almost surely, and set $r=\dim_{\mathbb R}V$.
The law of $H$, regarded as a Gaussian random vector on $V$, has a continuous density which is strictly positive at the origin. Since
$r\le \dim_{\mathbb R}\mathcal H_4=16$,
there exists $c>0$ such that, for all sufficiently small $\varepsilon>0$,
$$c\varepsilon^r
\le
\mathbb P\left(\|H\|_{HS}\leq\varepsilon\right)
\le
\mathbb P\left(\lambda_{\max}(H)\leq\varepsilon\right).$$
Since $r\le 16$, after decreasing $c$ if necessary,
$c\varepsilon^{16}
\le
\mathbb P\left(\lambda_{\max}(H)\le\varepsilon\right)
=
\mathbb P(F\le\varepsilon),$ 
for all sufficiently small $\varepsilon$. This contradicts
$\mathbb P(F\le\varepsilon)=O(\varepsilon^{17})$.
\end{proof}

A similar small-deviation argument would yield counterexamples in higher dimensions but, as shown next, it also provides counterexamples for three-state Markov chains. 

Let $\Sigma$ be a $3 \times 3$ covariance matrix such that $\Sigma \one = 0$ and let $B^\Sigma, X_j,$ and $F$ be defined as before except with $3$ instead of $4$ (see the definitions following equation \eqref{eq:sigmastar}). Let $C = (\cov (X_i(1), X_j(1)))_{1 \leq i, j \leq 2}$ and let $\rho = C_{12}/\sqrt{C_{11}C_{22}}$ when $C_{11}C_{22} > 0$.
\begin{thm}\label{thm:3counterexample}
If $-1 < \rho < -1/2$, there is no centered jointly Gaussian Hermitian random matrix $H \in {\cal H}_3$ such that $\lambda_{\max}(H)$ has the same law as $F$.
\end{thm}
\begin{proof}
Repeating the first two paragraphs of the proof of Theorem~\ref{thm:4counterexample} with $4$ replaced by $3$ and $\Sigma_*$  by $\Sigma$, we obtain
$F\geq\max_{1\le i\le 3}B_i^\Sigma(1)\ge 0$.
Moreover, on $\{F\le \varepsilon\}$,
$$ X_j(1)\ge -3\varepsilon
\qquad\text{and}\qquad
X_j(s)\le 3\varepsilon,
\quad 1\le j\le 2,\quad 0\le s\le 1.$$
Under the hypothesis $-1<\rho<-1/2$, $C$ is positive definite.  Therefore,
\begin{equation*}
W_\varepsilon(t) = C^{-1/2}
\left(
3\varepsilon\one_2-
\begin{pmatrix}
X_1(t)\\
X_2(t)
\end{pmatrix}
\right)
\end{equation*}
is a Brownian motion in $\mathbb R^2$ with identity covariance matrix, started at $3\varepsilon C^{-1/2}\one_2$. On $\{F\leq\varepsilon\}$,
$W_\varepsilon(1)\in\varepsilon C^{-1/2}[0,6]^2$,
and $W_\varepsilon$ remains in the closed wedge $\overline{C^{-1/2}[0,\infty)^2}$.  By the open–closed equivalence established after Lemma~\ref{lem:weyl-small-ball}, we may apply that lemma to its interior.
The opening of this wedge is $\theta=\arccos(-\rho)$.
Since $-1<\rho<-1/2$, we have $\theta<{\pi/3}$,
and therefore $$\alpha:=
2\left(1+\frac{\pi}{\theta}\right)=
2\left(
1+\frac{\pi}{\arccos(-\rho)}
\right) > 8.$$
Lemma~\ref{lem:weyl-small-ball} thus yields $\mathbb P(F\le\varepsilon)
= O(\varepsilon^\alpha)$, $\alpha>8$.

\noindent
Now suppose that $\lambda_{\max}(H)\stackrel{d}{=}F$ for some $3\times 3$ centered jointly Gaussian Hermitian random matrix $H$.
Since $F\ge 0$ almost surely, the linear support of $H$ cannot be all of
$\mathcal H_3$. Indeed, a nondegenerate centered Gaussian measure on
$\mathcal H_3$ has a strictly positive density everywhere and hence assigns positive probability to the nonempty open set of negative definite matrices, on which $\lambda_{\max}(H)<0$.
Thus, if $V$ denotes the linear support of $H$, $r:=\dim_{\mathbb R}V\le 9-1=8$.
As in the proof of Theorem~\ref{thm:4counterexample}, the Gaussian density on $V$ is positive at the origin, so there exists $c>0$ such that, for all sufficiently small $\varepsilon>0$,
$$ c\varepsilon^r\le \mathbb P\left(\|H\|_{HS}\le \varepsilon\right) 
\le
\mathbb P\left(\lambda_{\max}(H)\le \varepsilon\right).$$
Since $r\le 8$, this gives
$c\varepsilon^8
\le
\mathbb P(F\le\varepsilon)$ 
for all sufficiently small $\varepsilon$, contradicting
$\mathbb P(F\le \varepsilon)=O(\varepsilon^\alpha)$, $\alpha>8$.
\end{proof}

The following provides a concrete counterexample illustrating the previous result.  

\begin{example}
Consider the Markov chain on $1< 2< 3$ started from its uniform stationary distribution, with transition matrix
\begin{equation*}
\mathsf P = \frac{1}{6}
\begin{pmatrix}
1 & 1& 4 \\
1 & 4 & 1 \\
4 & 1 & 1
\end{pmatrix}.
\end{equation*}
$\mathsf P$ is symmetric and doubly stochastic with positive entries and eigenvalues ${1}/{2}$ and $-{1}/{2}$ on $\one^\perp$. 
Stationarity and symmetry give
\begin{equation*}
C = \frac{1}{9}
\begin{pmatrix}
14 & -13\\
-13 & 14
\end{pmatrix}
\end{equation*}
and hence $\rho = -13/14$. By Theorem~\ref{thm:3counterexample}, there is no centered jointly Gaussian Hermitian random matrix $H \in {\cal H}_3$ whose largest eigenvalue has the same law as $F$.
\end{example}

\begin{rem}\label{3cyclic}
For an irreducible, aperiodic, homogeneous cyclic chain on three states, it is shown in \cite[equation (3.12)]{HL2} that $F$ is, up to a scaling factor, the largest eigenvalue of a $3 \times 3$ traceless GUE random matrix. 
It is useful to examine this in light of Theorem~\ref{thm:3counterexample}.
The asymptotic covariance matrix $\Sigma$ is symmetric and cyclic, so since $\Sigma \one = 0$, it follows that $\Sigma = c(I_3 -\frac{1}{3} \one_3 \one_3^T)$ for some $c>0$, and so $C_{11} = C_{22}=-2C_{12}$ and $\rho = -1/2$. Therefore, the wedge in Theorem~\ref{thm:3counterexample} has opening $\pi/3$.  By Lemma~\ref{lem:weyl-small-ball},  $\PP (F \leq \varepsilon) = O(\varepsilon^8)$. 
Since the traceless subspace of ${\cal H}_3$ has real dimension $8$, the small-deviation exponent now matches the dimension of the Gaussian support, and there is no contradiction.
\end{rem}

\section{Extensions to Non-Uniform Stationary Distributions}

The preceding applications to Markov random words focused primarily on
the longest increasing subsequence and on chains with uniform stationary
distribution.  We now extend the random-matrix representation to the joint
fluctuations of all row lengths of the RSK Young 
diagram for a general stationary
distribution.

Grouping the letters according to their common stationary
probabilities, we first treat the top row and then show that, under suitable
blockwise isotropy and orthogonality conditions on the asymptotic covariance
matrix, the limiting row fluctuations are jointly represented by the
eigenvalues of a corresponding family of Gaussian Hermitian matrices.  
These results are then used to extend these limiting laws to the infinite countable alphabet case through a sandwiching procedure.

Let $(X_n)_{n\ge 0}$ be an irreducible, aperiodic, homogeneous Markov chain on the totally ordered alphabet $1 < \cdots< M$, started from its stationary distribution $\pi = (\pi_1, \ldots, \pi_M)^T$, that is $\mathbb P(X_0=k) =\pi_k, k=1, 2, \dots, M$, and with asymptotic covariance matrix $\Sigma$. Let 
$p_1> \cdots > p_K$  be the distinct values taken by the
stationary probabilities $\pi_1,\ldots,\pi_M$. For $1 \leq a \leq K$, let ${\cal A}_a= \{i : \pi_i = p_a \}$, $d_a = \abs{{\cal A}_a}$, and $m_a = \sum_{b< a} d_b$, where we adopt the convention that the empty sum is zero, 
and list the elements of each ${\cal A}_a$ in the order inherited from the
alphabet.  Next, let $\Sigma_{ab}$ denote the ${\cal A}_a\times {\cal A}_b$  submatrix of $\Sigma$ (thus of size $d_a\times d_b$), and let
\begin{equation*}
Q_a = I_{d_a} - \frac{1}{d_a} \one_{d_a} \one_{d_a}^T.
\end{equation*}
For $1\le a\leq K$, let $H_a$ be a $d_a\times d_a$ Hermitian
random matrix. Assume that their concatenated diagonal vector satisfies
$\left(\mathbf d(H_1)^T,\ldots,\mathbf d(H_K)^T\right)^T
 \sim \mathcal N(0,(\Sigma_{ab})_{a,b=1}^K)$. 
Independently of this diagonal vector, suppose that all the
upper-triangular off-diagonal entries of the matrices $H_1,\ldots,H_K$ are mutually
independent and that $(H_a)_{ij}\sim\mathcal N(0,c_a)_{\mathbb C}$, $1\le i<j\le d_a$, 
where $c_a\ge 0$.

Let $R_N^j$ denote the $j$th row length of the Young diagram obtained
by applying the RSK correspondence to the Markov word, and recall that
$R_N^1=LI_N$. (We set $R_N^j=0$ when the diagram has fewer than $j$ rows.)

Writing ${\cal A}_a=\{i_{a,1}<\cdots<i_{a,d_a}\}$,
set $B^{\Sigma,{\cal A}_a} := \left(B^\Sigma_{i_{a,1}},\dots,B^\Sigma_{i_{a,d_a}}\right)^T$.  
For the first block, define
$$F_1 :=\max_{0=t_0\leq\cdots\leq t_{d_1}=1}\sum_{j=1}^{d_1}\left(B^\Sigma_{i_{1,j}}(t_j)-B^\Sigma_{i_{1,j}}(t_{j-1})\right).$$

We begin with the top row, whose limiting law depends only on the covariance block corresponding to the largest stationary probability. 

\begin{lemma}\label{lem:smallresult}
Let $Q_1 \Sigma_{11} Q_1 = c_1 Q_1$. Then, as $N \to \infty$,
\begin{equation*}
\frac{R_N^1 - Np_1}{\sqrt{N}} \Longrightarrow \lambda_{\textrm{max}} (H_1).
\end{equation*}
\end{lemma}
\begin{proof}
If $d_1 = 1$, then the result follows immediately from the top row case of \cite[Theorem 3.2]{HL2}.

Let $d_1 \geq 2$ and $H_0 \in \mathcal M_{I_{d_1}}$. Then, choose $Z$ as in the final assertion of Theorem~\ref{thm:qsigmaq} applied to $\Sigma_{11}$, and set
\begin{equation*}
\tilde{H_1} = \sqrt{c_1} \left ( H_0 - \frac{\Tr(H_0)}{d_1}I_{d_1} \right ) + ZI_{d_1}.  
\end{equation*}

By the top-row case of \cite[Theorem 3.2]{HL2},
$\frac{R_N^1-Np_1}{\sqrt N}\Longrightarrow F_1$.  
It therefore remains to identify the law of $F_1$.  First, 
\begin{equation*}
(B^{\Sigma, {\cal A}_1}(1), F_1) \stackrel{d}{=} \left ( \mathbf d({\tilde H_1}), \lambda_{\textrm{max}} (\tilde H_1) \right ),
\end{equation*}
and in particular,
\begin{equation*}
\mathbf d (\tilde H_1) = \sqrt{c_1} Q_1 \mathbf d (H_0) + Z \one_{d_1} \sim \mathcal N(0, \Sigma_{11}).
\end{equation*}
The matrices $\widetilde H_1$ and $H_1$ have the same diagonal law,
and their off-diagonal entries are independent centered complex Gaussian
random variables with variance $c_1$, independent of their diagonals.
Therefore, $\widetilde H_1\overset d=H_1$, and $F_1\overset d=\lambda_{\max}(H_1)$. \end{proof}

We now extend the top-row identification to the joint fluctuations of all rows, under blockwise isotropy and independence of the projected Brownian bridges across blocks.  This establishes the corresponding direct-sum result for the full Young diagram.

\begin{thm}\label{thm:directsumversion}
Let the covariance matrix $\Sigma$ be such that 
$$
 Q_a\Sigma_{aa}Q_a=c_aQ_a,\qquad c_a\ge 0,
 \qquad 1\le a\le K,
$$
and
$$
 Q_a\Sigma_{ab}Q_b=0,
 \qquad 1\le a<b\le K.
$$
Let $\mu_1(H_a) \ge \cdots \ge \mu_{d_a}(H_a)$ be the ordered eigenvalues of $H_a$. Then, as $N \to \infty$,
\begin{equation*}
\bigoplus_{a=1}^K \left ( \frac{R_N^{m_a +r} - Np_a}{\sqrt{N}} \right )_{r=1}^{d_a} {\Longrightarrow} \ \bigoplus_{a=1}^K \left (\mu_r(H_a) \right )_{r=1}^{d_a}.
\end{equation*}
\end{thm}
\begin{proof}
Observe that $B^\Sigma(t)- tB^\Sigma(1)$ is independent of $B^\Sigma(1)$, and 
for $0\le s,t\le 1$,
\begin{eqnarray*}
&\operatorname{Cov}\left(
 Q_a\left(B^{\Sigma,{\cal A}_a}(s)-sB^{\Sigma,{\cal A}_a}(1)\right),
 Q_b\left(B^{\Sigma,{\cal A}_b}(t)-tB^{\Sigma, {\cal A}_b}(1)\right)
 \right) \\
&
=(s\wedge t-st)Q_a\Sigma_{ab}Q_b.
\end{eqnarray*}
Since the projected bridges are jointly Gaussian, the condition
$Q_a\Sigma_{ab}Q_b=0$ makes the bridges belonging to distinct blocks
independent, and within each block the covariance is $(s \wedge t -st) c_a Q_a$. Adding the same scalar bridge to every coordinate does not change any of
the path functionals, since its contribution along each path from $0$ to
$1$ telescopes to zero.  So the $r$-path functional in \cite[equation (3.8)]{HL2} depends on the bridge only through $Q_a\left (B^{\Sigma, {\cal A}_a}(t) -tB^{\Sigma, {\cal A}_a}(1) \right)$.  Thus, for $c_a>0$, conditionally on $B^{\Sigma,{\cal A}_a}(1)$, the block-$a$ path functionals have the same joint law as the corresponding functionals of a Brownian motion with covariance matrix $c_aI_{d_a}$, conditioned to have the same endpoint.

If $c_a=0$, then
$Q_aB^{\Sigma,{\cal A}_a}(t)=0$ almost surely for every $t$, so all the coordinates
of $B^{\Sigma,{\cal A}_a}$ are identical, and the off-diagonal entries of $H_a$ are zero. Then, every path sum telescopes to $(B^{\Sigma, {\cal A}_a}(1))_1$, so the $r$-path functional is $r(B^{\Sigma, {\cal A}_a}(1))_1$. Since the projected bridges and off-diagonal entries are independent across distinct blocks, these conditional identities hold across all blocks given $B^{\Sigma}(1)$. 
    
For $c_a>0$, first consider the block identity when $\Sigma_{aa}=c_aI_{d_a}$; scaling the identity in \cite[Theorem 2.1]{BGH} by $\sqrt{c_a}$ identifies the full array of path functionals with the partial sums $(\sum_{i=1}^r \mu_i(H_a^{(k)}) )_{1 \leq r \leq k \leq d_a}$ of the eigenvalues of $H_a^{(k)}$, where $H_a^{(k)}$ denotes the $k\times k$ leading principal minor of $H_a$. 
Since
$\operatorname{Tr}\left(H_a^{(k)}\right)$ corresponds under this joint array identity to $\sum_{i=1}^k (B^{\Sigma,{\cal A}_a}(1))_i$, 
taking consecutive differences in $k$ recovers both
$\mathbf d(H_a)$ and $B^{\Sigma,{\cal A}_a}(1)$, as in Lemma~\ref{lem:summationidentity}.  Conditioning on the diagonal vector, or equivalently on the endpoint, gives the identity for almost every endpoint. Since the endpoint in the GUE minor identity has positive density on $\RR^{d_a}$, and the path functionals and partial eigenvalue sums are Lipschitz in the endpoint, the identity holds for every endpoint. 
Conditionally on $B^{\Sigma,{\cal A}_a}(1)$, the projected bridge
$Q_a\left(B^{\Sigma,{\cal A}_a}(t)-tB^{\Sigma,{\cal A}_a}(1)\right)$ 
has the same law as the corresponding projected bridge in the isotropic
model with covariance matrix $c_aI_{d_a}$. Together with the
independence and Gaussian law of the off-diagonal entries of $H_a$, this
extends the conditional identity to the covariance matrix
$\Sigma_{aa}$.

By \cite[Theorem 3.2]{HL2}, the centered and scaled partial sums of the RSK row lengths converge
jointly. Hence, as $N\to +\infty$, 
\begin{eqnarray*}
&\bigoplus_{a=1}^K \left ( \frac{\sum_{j=1}^{m_a + r} R_N^j - N \left ( \sum_{b< a} d_b p_b + r p_a \right )}{\sqrt{N}} \right )_{r=1}^{d_a} \\
&\qquad \qquad \qquad \qquad \Longrightarrow \bigoplus_{a=1}^K \left ( \sum_{b < a} \Tr(H_b) + \sum_{j=1}^r \mu_j(H_a) \right )_{r=1}^{d_a}.
\end{eqnarray*}

To get the individual row lengths, take consecutive differences in the lexicographic ordering of the pairs $(a, r)$.  We consider the cases. For $r>1$, subtracting the $(a, r-1)$ component  gives
\begin{equation*}
\frac{R_N^{m_a + r}- N p_a}{\sqrt{N}} \Longrightarrow \sum_{j=1}^r \mu_j(H_a) -\sum_{j=1}^{r-1} \mu_j(H_a) = \mu_r (H_a).  
\end{equation*} 
For $r=1$ and $a>1$, subtracting the last sum of block $a-1$, with the usual empty-sum convention when $a=2$, and using $\sum_{j=1}^{d_{a-1}} \mu_j(H_{a-1}) = \Tr(H_{a-1})$ gives, as $N\to +\infty$, 
\begin{eqnarray*}
\frac{R_N^{m_a + 1} -N p_a}{\sqrt{N}} \Longrightarrow &\sum_{b< a} \Tr (H_b) + \mu_1 (H_a) - \left ( \sum_{b < a-1} \Tr (H_b) + \Tr(H_{a-1}) \right ) \\ 
&= \mu_1 (H_a).
\end{eqnarray*}  
Finally, if $r=1$ and $a=1$, then the leading term shows that, as $N\to +\infty$, 
\begin{equation*}
\frac{R_N^1 - Np_1}{\sqrt{N}} \Longrightarrow \mu_1 (H_1).
\end{equation*}
The map taking the $M$ individual row lengths to their successive
partial sums, ordered block by block, is a continuous invertible linear map
on $\RR^M$.  The continuous mapping theorem therefore yields the 
claimed joint convergence of the individual row lengths.  
\end{proof}

Finally, we generalize our approach to the infinite countable alphabet $1 < 2< \cdots$. 
Let $(X_n)_{n \geq 0}$ be an irreducible, aperiodic, positive recurrent, homogeneous Markov chain started from its stationary distribution $\pi = (\pi_i)_{i \ge 1}$. 

Let $p_1>p_2>\cdots $ be the distinct values assumed by the stationary
probabilities $(\pi_i)_{i\ge 1}$ , and set
$$ {\cal A}_a:=\{i\geq1:\pi_i=p_a\},\qquad
d_a:=|{\cal A}_a|,\qquad
m_a:=\sum_{b<a}d_b, 
$$
with the empty-sum convention $m_1 = 0$.  Clearly ${\cal A}_a$ is finite since $p_a>0$ and $\sum_i\pi_i=1$.  
For $i \ge 1$, define $S_N^i$ at the grid points by
\begin{equation*}
S_N^i\left(\frac{k}{N}\right) = \frac{1}{\sqrt{N}} \sum_{\ell=1}^k \left ( \one_{\{X_\ell = i\}} - \pi_i \right ), \qquad 0\leq k\leq N,
\end{equation*}
and extend $S_N^i$ to $[0,1]$ by linear interpolation.
Let $B^\Sigma=(B_i^\Sigma)_{i\ge 1}$ be a centered Gaussian process such that 
$$\operatorname{Cov}\left(B_i^\Sigma(s),B_j^\Sigma(t)\right)
 =(s\wedge t)\Sigma(i,j),
 \qquad i,j\ge 1,\quad 0\le s,t\le 1.
$$
Assume that, for every $L\ge 1$,
$$(S_N^1,\dots,S_N^L) \stackrel{N\to +\infty}\Longrightarrow (B_1^\Sigma,\ldots,B_L^\Sigma),$$ 
in $C([0,1],\mathbb R^L)$.  

Now, write ${\cal A}_a=\{i_{a,1}<\cdots<i_{a,d_a}\}$. For $1\le r\le d_a$, define
$F_{a,r}^\Sigma$ to be the path functional of $B^{\Sigma, {\cal A}_a}$ used in the proof of Theorem~\ref{thm:directsumversion}, namely, 
$$
F_{a,r}^\Sigma:=
\max_{\mathbf t\in\mathcal I_{r,d_a}}
\sum_{j=1}^{r}\sum_{\ell=j}^{d_a-r+j}
\left[
B^\Sigma_{i_{a,\ell}}(t_{j,\ell})
-B^\Sigma_{i_{a,\ell}}(t_{j,\ell-1})
\right],
\qquad F_{a,0}^\Sigma:=0,
$$
where $\mathcal I_{r,d}$ consists of arrays $\left(t_{j,\ell}\right)_{1\le j\le r, 0\le\ell\le d}$ such that $t_{j,\ell}=0, 0\le\ell\le j-1$, 
$t_{j,\ell}=1, d-r+j\le\ell\le d$, 
$t_{j,\ell-1}\le t_{j,\ell}, 1\le j \le r, 1\le \ell \le d$, $t_{j+1,\ell+1}\le t_{j,\ell}, 1\le j<r, 0\le\ell<d$. 
The boundary conditions give $F_{a,d_a}^\Sigma=\sum_{\ell=1}^{d_a}B^\Sigma_{i_{a,\ell}}(1)$, as used later in the proof below.  $F_{a,r}^\Sigma$ is defined as the $r$-path functional of the coordinates indexed by ${\cal A}_a$, using the admissible-time set and 
\cite[equation (3.8)]{HL2}.  The corresponding limit in \cite[equation (3.8)]{HL2} also includes endpoint terms from earlier blocks; these terms are not included in $F_{a,r}^{\Sigma}$.

\begin{thm}\label{thm:countable}
Under the assumptions above, for every fixed $a \ge 1$, 
\begin{equation*}
\bigoplus_{b=1}^a \left( \frac{R_N^{m_b+r}-Np_b}{\sqrt{N}} \right)_{r=1}^{d_b} 
\stackrel{N\to +\infty}\Longrightarrow \bigoplus_{b=1}^a \left ( F_{b,r}^\Sigma -F_{b,r-1}^\Sigma \right )_{r=1}^{d_b}.
\end{equation*}
Moreover, as $N \to +\infty$, these finite-block convergences hold jointly as a convergence of the entire sequence:
\begin{equation*}
\left[
\left(
\frac{R_N^{m_b+r}-Np_b}{\sqrt N}
\right)_{r=1}^{d_b}
\right]_{b\ge1}
\ \Longrightarrow\
\left[
\left(F_{b,r}^\Sigma-F_{b,r-1}^\Sigma\right)_{r=1}^{d_b}
\right]_{b\ge1},
\end{equation*}
where the convergence holds in $\prod_{b\ge1}\mathbb R^{d_b}$ equipped with its product topology. 

Fix $a\ge 1$.  For $b,c\ge 1$, set $$\Sigma_{bc}:=\left(\Sigma(i_{b,k},i_{c,\ell})\right)_
{1\le k\le d_b,\;1\le\ell\le d_c},
\qquad Q_b:=I_{d_b}-\frac1{d_b}\mathbf1_{d_b}\mathbf1_{d_b}^{T}.$$
In addition, suppose there exist constants $c_b\ge 0$, $1\le b\le a$, such that
$Q_b\Sigma_{bb}Q_b=c_bQ_b, 1\le b\le a$, 
$Q_b\Sigma_{bc}Q_c=0, 1\le b<c\le a$. 

Let $(H_1,\ldots,H_a)$ have centered jointly Gaussian diagonals with block covariance $(\Sigma_{bc})_{1\le b,c\le a}$.  
Independently of these diagonal vectors, let all upper-triangular off-diagonal entries, across all matrices, be mutually independent, with
$(H_b)_{ij}\sim N(0,c_b)_{\mathbb C}$, $1\le i<j\le d_b$.  Then the limiting vector for the first $a$ blocks has the same law as $\bigoplus_{b=1}^{a}
\left(\mu_r(H_b)\right)_{r=1}^{d_b}$.

If the displayed conditions hold for every block and every pair of blocks, the same identification holds for the entire limiting sequence, with  $(H_b)_{b\ge1}$ constructed consistently.
\end{thm}

\begin{proof}
Fix $a \ge 1$, let $N$ be the original word length, and choose $M$ such that $\sum_{i>M} \pi_i < p_a$.  This choice of $M$ and the corresponding inequality ensure that all letters in ${\cal A}_1, \ldots, {\cal A}_a$
lie in $\{1,\ldots,M\}$, while the merged letter $M+1$ has a strictly smaller stationary frequency than every letter in those blocks.

Let $R_N^{j, -}$  be the $j^{\textrm{th}}$ row length obtained by removing all letters greater than $M$ from $X_1, \ldots, X_N$, and let $R_N^{j,+}$ be the row length obtained by replacing each $X_n$ by $X_n \wedge (M+1)$. 
(We also set $R_N^{j,\pm}=0$ when the diagram has fewer than $j$ rows.)
The map $i\mapsto i\wedge(M+1)$ is nondecreasing, so it preserves every weakly increasing subsequence.  Moreover, deleting letters cannot create a longer collection of disjoint subsequences. (Such procedures are at the root of the reduced and capped sequences in \cite{HL0}.)

After removing letters, count the remaining letters according to their positions in the original word. 
Keeping the deleted positions empty leaves the path formula for the remaining word unchanged, so whether one removes letters or applies the map $X_n \mapsto X_n \wedge (M+1)$, the interpolated count of letter $i \leq M$ is $N \pi_i t + \sqrt{N} S_N^i(t)$.  After replacement, the count of letter $M+1$ is $Nt \sum_{i>M} \pi_i - \sqrt{N} \sum_{i=1}^M S_N^i(t)$, and centering and rescaling gives $S_N^1, \ldots, S_N^M$ with $-\sum_{i=1}^M S_N^i$ for letter $M+1$ after replacement.

For $b \leq a$ and $1 \leq r \leq d_b$, apply the deterministic count-path representation of \cite[Theorem 3.1, equation (3.1)]{HL2} to the sum $\sum_{j=1}^{m_b+r} R_N^{j,\pm}$.  Replacing the letter counts by their expectations gives a continuous function of the admissible times, with maximum $N(\sum_{\ell< b} d_{\ell} p_{\ell} + rp_b)$. After centering and scaling, the path formula maximizes a sum of increments of $S_N^1, \ldots, S_N^M$, with $-\sum_{i=1}^M S_N^i$ for letter $M+1$, plus a nonpositive deterministic term.

Now, the reduced and capped words need not themselves be Markov chains.  However, \cite[Theorem 3.2]{HL2} contains two logically different ingredients: a deterministic analysis of the path formula and a functional CLT for its particular Markov chain. Here the first ingredient is needed, but the second is supplied by the stated CLT for the original count paths. Deleting or capping letters does not require a new Markov-chain CLT. The zero-drift reduction is worked out in the proof of Theorem 3.2 in \cite{HL2}.  We analyze next why the finite-alphabet path-formula argument applies to both truncations. 

Let $q_M:=\sum_{i>M}\pi_i$, 
$\nu_{b,r}=\sum_{\ell<b}d_\ell p_\ell+rp_b$, and $k=m_b+r$.  
For either the reduced or the capped word, the path formula for the sum of its first $k$ row lengths has a compact admissible-time set $\mathcal T_{b,r}^{\pm}$. After subtracting $N\nu_{b,r}$ and dividing by $\sqrt N$, its objective takes the form
$$\sqrt N h_{b,r}^{\pm}(\mathbf t)
+g_{N,b,r}^{\pm}(\mathbf t), \qquad \mathbf t\in\mathcal T_{b,r}^{\pm}.$$

Here $h^\pm_{b,r}$ is obtained by replacing each normalized count path by its stationary-frequency term and subtracting $\nu_{b,r}$, while $g^\pm_{N,b,r}$  contains the increments of the centered count paths.  Since $\nu_{b,r}$ is the sum of the $k$ largest letter frequencies in either truncated alphabet, $h_{b,r}^{\pm}\le 0$ and $\max h_{b,r}^{\pm}=0$. In the capped alphabet the additional frequency is $q_M<p_a\leq p_b$, so this extra letter cannot receive positive time in a zero-drift path for $b\le a$, since 
$-h_{b,r}^{+}(\mathbf t)\ge (p_b-q_M)D(\mathbf t)$,
where $D(\mathbf t)$ is the time assigned to that letter.

For completeness, if $\mathcal T$ is compact, $h\le 0$ is continuous with a nonempty zero set, and as $N\to +\infty$, $g_N\to g$ uniformly, with $g_N$ continuous, then
$$\max_{\mathbf t\in\mathcal T}
{\sqrt N\,h(\mathbf t)+g_N(\mathbf t)}
\longrightarrow
\max_{\mathbf t:\,h(\mathbf t)=0}g(\mathbf t).$$
Indeed, comparison with a point where $h=0$ forces every maximizing sequence to approach the zero set; compactness gives the upper bound, and restricting the maximum to the zero set gives the lower bound.  
For each truncation, the deterministic zero-drift analysis in the proof of \cite[Theorem 3.2]{HL2} identifies the maximum of the limiting centered path functional restricted to this zero set as
$$\sum_{\ell<b}\mathbf1_{d_\ell}^{T}B^{\Sigma,{\cal A}_\ell}(1)
+F_{b,r}^{\Sigma}.$$
For $b\le a$ and $1\le r \le d_b$, set
$$
V_{N,b,r}^{\pm}
:=
N^{-1/2}\left(\sum_{j=1}^{m_b+r}R_N^{j,\pm}
-N\left(\sum_{\ell<b}d_\ell p_\ell+rp_b\right)\right),$$
and 
$$Z_{b,r}:=
\sum_{\ell<b}\mathbf1_{d_\ell}^{T}B^{\Sigma,{\cal A}_\ell}(1)
+F_{b,r}^{\Sigma}.
$$
Let $V_N^\pm=(V_{N,b,r}^\pm)_{b\le a,1\le r\le d_b}$ and $\widetilde Z=(Z_{b,r})_{b\le  a, 1\le r\le d_b}$. The centered count paths for the reduced word are $S_N^1,\dots,S_N^M$ while those for the capped word are these same paths together with $-\sum_{i=1}^M S_N^i$. Hence both $V_N^-$ and $V_N^+$ are determined by the same vector of original count paths. Applying the zero-drift localization argument simultaneously to the finitely many path functionals, the zero-drift restriction for either truncation is the same functional $\widetilde Z$ of the same limiting Brownian paths. 
The assumed joint functional CLT therefore yields almost-sure uniform convergence of the finitely many random terms $g^\pm_{N,b,r}$ to their limiting Brownian objectives under a Skorokhod coupling, allowing the stated deterministic lemma to be applied simultaneously. Hence, 
$$
(V_N^-,V_N^+)\ \stackrel{N\to+\infty}\Longrightarrow\ (\widetilde Z, \widetilde Z).
$$
By the continuous mapping theorem, $V_N^+-V_N^-\to 0$ in probability. Since the pairs $(b,r)$, in lexicographic order, correspond to $j=1,\dots,m_a+d_a$, this says precisely that, as $N \to +\infty$, 
$$
\max_{1\le j\leq m_a+d_a}N^{-1/2}\left(\sum_{i=1}^{j}
\left(R_N^{i,+}-R_N^{i,-}\right)\right) \stackrel{\mathbb P}\longrightarrow 0.$$

The maximal total length of $j$ disjoint weakly increasing subsequences cannot be increased by removing letters and cannot be decreased by the map $X_n \mapsto X_n \wedge (M+1)$. Hence, by the characterization of RSK row sums as maximal total lengths of disjoint weakly increasing subsequences (see \cite[Section 3.1, Lemma 1]{Fulton}), for every $j \geq 1$,
\begin{equation*}
\sum_{i=1}^j R_N^{i,-} \leq \sum_{i=1}^j R_N^i \leq \sum_{i=1}^j R_N^{i,+}.
\end{equation*} Thus, the centered and scaled partial sums of the original row lengths have the same joint limit. As in the proof of Theorem~\ref{thm:directsumversion}, take consecutive differences. Using $F_{b,d_b}^\Sigma = \one_{d_b}^T B^{\Sigma, {\cal A}_b}(1)$ gives the claimed convergence for the first $a$ blocks.

Since $a$ was arbitrary, for each $b$, finite-block convergence gives tightness of that block’s laws, uniformly in $N$, after enlarging the compact sets to accommodate finitely many initial values of $N$. Choose compact sets $K_b\subset\mathbb R^{d_b}$, with $\sup_N\mathbb P(Y_{N,b}\notin K_b)<\varepsilon2^{-b}$, 
where $Y_{N,b}:=
\left({(R_N^{m_b+r}-Np_b)}/{\sqrt N}
\right)_{r=1}^{d_b}
\in\mathbb R^{d_b}$. 
Then, for the compact set (in the product topology) $\prod_{b\ge 1} K_b$, and setting $Y_N := (Y_{N,b})_{b\ge 1}$, one has $\inf_N\mathbb P\!\left(Y_N\in\prod_{b\ge1}K_b\right)\ge 1-\varepsilon$.  Every subsequential limit has the finite-block distributions identified above. Since these distributions determine a unique probability law on $\prod_{b\geq1}\mathbb R^{d_b}$, the entire sequence converges to the claimed limit.

If the first $a$ blocks satisfy the covariance conditions, the argument with projected bridges in the proof of Theorem~\ref{thm:directsumversion} applied to these blocks shows that the limiting vector is distributed as the ordered eigenvalues of $H_1, \ldots, H_a$.  For the full-sequence assertion, note that the joint Gaussian diagonal vectors for all blocks can be defined consistently from the covariance kernel; their finite-dimensional laws then determine the infinite eigenvalue sequence.
\end{proof}

\begin{rem}\label{rem:final}
We conclude this manuscript by indicating several directions for further investigation.
The foremost open problem is to identify the law of the Brownian
last-passage functional $F_\Sigma$ beyond the isotropic regime, where the
counterexamples of Section~5 show that a Gaussian maximal-eigenvalue
representation need not exist. It would also be natural to seek
process-level versions of the random-matrix identifications and Young
diagram limits established here.  We plan to pursue this elsewhere. 
Queueing and interacting-particle
interpretations of these results may provide further insight into the role
of the covariance structure and suggest alternative constructions of the
limiting processes.  The study of longest increasing subsequences in random words is closely connected to the corresponding problem for random permutations. Determining analogous limiting laws for permutations generated by Markovian models appears more elusive still.  
\end{rem}

\begin{funding}
The first author was supported in part by grants \#524678 and MP-TSM-00002660 from the Simons Foundation.
\end{funding}

\bibliographystyle{spr-ims-number}
\bibliography{Bibliography}

\end{document}